\documentclass[11pt]{article}

\usepackage[utf8]{inputenc}
\usepackage[T1]{fontenc}
\usepackage[main=english,ngerman]{babel}
\usepackage{lmodern}
\usepackage{tcolorbox}
\usepackage{enumerate}
\usepackage{multirow}
\usepackage{nicefrac}
\usepackage{caption}
\usepackage{graphicx}
\usepackage{tikz}
\usetikzlibrary{patterns}
\usetikzlibrary{arrows}
\usetikzlibrary{decorations}
\usetikzlibrary{decorations.pathreplacing}
\usepackage{tikz-3dplot}
\usepackage{subfigure}
\usepackage{pgfplots}
\usepackage{hyperref}
\usepackage{picture}
\usepackage[absolute,overlay]{textpos}

\usepackage[
left=3cm,
right=3cm,
top=2.8cm,
bottom=2.7cm,
]{geometry}
\usepackage[affil-it]{authblk}
\usepackage{media9}

\usepackage{amsmath,amssymb,amsthm}
\usepackage{mathtools}
\usepackage[mathscr]{euscript}
\usepackage{xpatch}
\usepackage{asymptote}
\usepackage{mhchem}

\newtheorem{definition}{Definition}
\theoremstyle{remark}
\newtheorem{remark}{Remark}
\theoremstyle{theorem}
\newtheorem{theorem}{Theorem}
\theoremstyle{theorem}
\newtheorem{lemma}{Lemma}

\newtheorem{corollary}{Corollary}
\newtheorem{assumption}{Assumption}

\usepackage[rightcaption]{sidecap}

\newcommand{\p}{\bar{p}}

\renewcommand{\p}{\textbf}

\newcommand\norm[1]{\left\lVert#1\right\rVert}
\graphicspath{{images/}}
\begin{document}
	
	\title{Structure-preserving Discretization of the Hessian Complex based on Spline Spaces}
	
	\author{Jeremias Arf and Bernd Simeon \thanks{E-mail: \texttt{\{arf, simeon\}@mathematik.uni-kl.de, \noindent \newline \noindent \ \ \emph{2020  Mathematics Subject Classification. \ 65N30; 65N12} , \\ \emph{Keywords.} \textup{Finite Element Exterior Calculus, B-splines, Hodge-Laplacian, Hilbert complexes, structure-preservation. }}}}	   
	    \affil{TU Kaiserslautern, Germany}
	   
	\date{}
    \maketitle

    \begin{abstract}
    	We want to propose a new discretization ansatz for the second order  Hessian complex exploiting benefits of isogeometric analysis, namely the possibility of high-order convergence and smoothness of test functions. Although our approach is firstly only valid in domains that are obtained by affine linear transformations of a unit cube, we see in the approach a relatively simple way to obtain inf-sup stable and arbitrary fast convergent methods for the underlying Hodge-Laplacians. Background for this is the theory of Finite Element Exterior Calculus (FEEC) which guides us to structure-preserving discrete sub-complexes.  
    \end{abstract}
\section{Introduction}	

The Hessian complex is a so-called Hilbert complex that pops up in different fields like numerical relativity (\cite{QuennevilleBlair2015ANA}) or also as underlying complex for a biharmonic problem; see \cite{Pauly2016OnCA}. A more famous Hilbert complex in numerical mathematics and physics is the first order de Rham complex

		\begin{tikzpicture}
				\node at (-2,1.4) {$\mathbb{R}$};
		\node at (0,1.4) {$H^1(\Omega)$};
		
		\node at (3,1.4) {$\p{H}(\Omega,\textup{curl})$};

		\node at (6.5,1.4) {$\p{H}(\Omega,\textup{div})$};

		\node at (9.5,1.4) {$L^2(\Omega)$};

		\node at (1.3,1.7) {$\nabla$};
		\node at (-1.3,1.65) {$\subset$};
		
		\node at (4.7,1.7) {$\nabla \times $};

		\node at (8.2,1.7) {$\nabla \cdot $};

		\draw[->] (0.8,1.4) to (1.8,1.4);
		
		\draw[->] (-1.7,1.4) to (-0.7,1.4);
		\draw[->] (4.2,1.4) to (5.3,1.4);
		
		\draw[->] (7.7,1.4) to (8.75,1.4);
		\draw[->] (10.2,1.4) to (10.9,1.4);
		\node at (10.55,1.7) {$0$};

		\node at (11.45,1.4) {$\{0\} \ ,$};
	
		\end{tikzpicture}
	\noindent \\
 which  can be used for problems in electromagnetics; see \cite{Buffa2011IsogeometricDD}. It  has  a connection to Maxwell's equations; compare \cite[section 8.6]{ArnoldBook}. An elegant way of  discretizing latter Hilbert complex by setting up a finite-dimensional subcomplex can be achieved through the theory of FEEC developed mainly by  Arnold,  Falk and  Winther; see e.g. \cite{Arnold2010FiniteEE}. One of the basic ideas behind FEEC is the usage of test function spaces which are compatible with the complex in the sense that we have projections onto the finite-dimensional spaces that commute with the differential operators. In his book \cite{ArnoldBook} Arnold introduces in detail how one can construct Finite Element (FEM) spaces fulfilling the commutation and other properties and in what way they lead to stable and convergent numerical variational formulations for different kind of equations, e.g. the Hodge-Laplacians. Underlying are the spaces of polynomial differential forms. 
Fortunately, the results of FEEC are  quite general and  the framework is applicable for every closed Hilbert complex and other types of discrete spaces. For example  Buffa et al.  presented in the paper  \cite{Buffa2011IsogeometricDD} a procedure of discretizing for the de Rham complex using spline spaces that satisfy  the main aspects of FEEC. \\
Here in this article we want to continue the idea of combining isogeometric analysis (IGA) and FEEC within the scope of numerical methods for the example of the Hessian complex. In other words we adapt the approach in \cite{Buffa2011IsogeometricDD} for the case of the Hessian complex and orient ourselves very closely towards latter reference.

A main reason for our studies is the search for a stable  and convergent  numerical method for the Linearized Einstein Bianchi System (LEBS) in numerical relativity.
For a derivation of the LEBS we refer the reader to the thesis \cite{QuennevilleBlair2015ANA} of Quenneville-B\'elair  and the references therein. Further, the author of \cite{QuennevilleBlair2015ANA}  uses the concept of FEEC for the numerical computation of solutions to the LEBS, too. Hence we use a similar mind walk since we are also looking for structure-preserving discretizations using the results of FEEC.  But  whereas Quenneville-B\'elair uses polynomial de Rham complexes, we exploit isogeometric analysis for the definition of test function spaces. Because of the possibility to increase the smoothness of splines easily we are able to discretize the original Hessian complex with its required $C^1$ regularity of the test functions due to fact  that $H^2$-Sobolev spaces are involved. Furthermore, as Quenneville-B\'elair pointed out in his thesis, the version of the LEB system as a part of the Hessian complex guarantees automatically some special features of the physics behind the equations. Namely, suitable symmetry and trace properties are fulfilled, or preserved, respectively. Thus  one of the outcomes of this article is the achievement of a stable high-order convergent method for the Hessian complex  that is feasible for an application in the context of numerical relativity. However, the needed restriction to affine linear parametrizations for the  proposed method demonstrates the meaningfulness of generalizations. Especially the study of the Hessian complex on  geometries with curved boundaries is of current interest for the authors.\\

We also want to mention that the idea of using splines for Hilbert complexes like presented in \cite{Buffa2011IsogeometricDD}  should also be applicable in the context of other tensor complexes, e.g. the \textup{elasticity}-complex; see \cite{Pauly2016OnCA} or \cite[Chapter 8]{ArnoldBook}.  \\

The paper is structured as follows. In Section 2 we introduce mathematical notation and basic notions in the context of isogeometric analysis as well as for Hilbert complexes. Afterwards, we define a discrete Hessian complex using splines. Then we face approximation estimates for quantifying the  goodness of the discretization. In Section 5 we introduce two application examples, namely the Hodge-Laplacian and the LEBS. In the last Section 6 we display some numerical tests for checking the convergence statements established before in Section 4.

\section{Mathematical preliminaries and notation}
\subsection{Mathematical notation}
In this section we introduce some notation and define several spaces. \\
Given some bounded Lipschitz domain $D \subset \mathbb{R}^3$ we write for the standard Sobolev spaces $H^0(D)=L^2(D),\ H^k(D), \ k \in \mathbb{N}$, where $L^2(D)$ stands for the Hilbert space of square-integrable functions endowed with the inner product $\langle \cdot , \cdot \rangle_{L^2(D)}$. The norms $(|\cdot|_{H^k(D)}),\ \norm{\cdot}_{H^k(D)}$ denote the classical Sobolev (semi-)norms in $H^k(D)$. 
 In case of vector- or matrix-valued functions we can define Sobolev spaces, too, by requiring the component functions to be in suitable Sobolev spaces. To distinguish latter case from the scalar-valued one, we use a bold-type notation. For example we have for $\boldsymbol{v} \coloneqq (v_1,\dots , v_d), \ \boldsymbol{v} \in \p{H}^k(D) \colon \Leftrightarrow \ v_i \in H^k(\Omega), \, \forall i  $ and $\p{M} \coloneqq \big(M_{ij}\big)_{i,j=1}^3, \ \p{M} \in \p{H}^k(D) \colon \Leftrightarrow \ M_{ij} \in H^k(\Omega), \, \forall i,j$ and define the norms
\begin{align*}
\norm{\boldsymbol{v}}_{\p{H}^k(D)}^2 \coloneqq \sum_i \norm{v_i}_{H^k(D)}^2, \ \ \ \  \ \ \ \norm{\p{M}}_{\p{H}^k(D)}^2 \coloneqq \sum_{i,j} \norm{M_{ij}}_{H^k(D)}^2.
\end{align*}
Analogously we can proceed in case of the semi-norms.
We note that the inner product $\langle \cdot , \cdot \rangle_{L^2(D)}$ introduces straightforwardly an inner product on $\p{L}^2(D)$. For the definition of the next  spaces and norms we follow partly \cite{Pauly2016OnCA} to introduce further notation. First, let us consider vector-valued mappings.  Then we set
\begin{alignat*}{3}
 \p{H}(D,\textup{curl}) &\coloneqq \{ \boldsymbol{v} \in \p{L}^2(D) \ | \ \nabla \times \boldsymbol{v} \in \p{L}^2(D)  \}, \  \hspace{0.3cm} && \norm{ \boldsymbol{v}}_{\p{H}(D,\textup{curl})}^2  \coloneqq \norm{\boldsymbol{v}}_{\p{L}^2(D)}^2 + \norm{\nabla \times \boldsymbol{v}}_{\p{L}^2(D)}^2,\\
 \p{H}(D,\textup{div}) &\coloneqq \{ \boldsymbol{v} \in \p{L}^2(D) \ | \ \nabla \cdot \boldsymbol{v} \in \p{L}^2(D)  \}, \   && \norm{ \boldsymbol{v}}_{\p{H}(D,\textup{div})}^2  \coloneqq \norm{\boldsymbol{v}}_{\p{L}^2(D)}^2 + \norm{\nabla \cdot \boldsymbol{v}}_{{L}^2(D)}^2.
\end{alignat*}
\normalsize
 Above we wrote $\nabla$  for the classical nabla operator and later will write $\nabla^2$ for the Hessian. The definitions for $\p{H}(D,\textup{curl}) , \ \p{H}(D,\textup{div}), $ can be generalized to the matrix setting by requiring that all the rows (as vector-valued mappings) are in the respective spaces. Here, the curl $\nabla \times$ and divergence $\nabla \cdot$ act row-wise, too. Furthermore, we denote the subspace of symmetric and traceless matrix-valued functions by
$$\p{L}^2(D,\mathbb{S}) \coloneqq \{ \p{M} \in \p{L}^2(D) \ | \ \p{M}^{T}=\p{M} \}, \ \ \ \ \  \p{L}^2(D,\mathbb{T}) \coloneqq \{ \p{M} \in \p{L}^2(D) \ | \  \textup{tr}(\p{M})=0 \},$$ and set
$$ \p{H}(D,\textup{curl}, \mathbb{S}) \coloneqq  \p{H}(D,\textup{curl}) \cap \p{L}^2(D,\mathbb{S}), \ \ \ \ \  \p{H}(D,\textup{div}, \mathbb{T}) \coloneqq  \p{H}(D,\textup{div}) \cap \p{L}^2(D,\mathbb{T}) .$$ 

Besides we define
\begin{alignat*}{3}
\small
\p{H}(D,\textup{symcurl},\mathbb{T}) &\coloneqq \{ \p{T} \in \p{L}^2(D,\mathbb{T}) \ | \ \nabla \times \p{T} + (\nabla \times \p{T})^T \in \p{L}^2(D) \}, \\
\p{H}(D,\textup{divdiv},\mathbb{S}) &\coloneqq \{ \p{S} \in \p{L}^2(D,\mathbb{S}) \ | \ \nabla \cdot \nabla \cdot \p{S} \in {L}^2(D) \},\\
\norm{ \p{M}}_{\p{H}(D,\textup{divdiv})}^2  &\coloneqq \norm{\p{M}}_{\p{L}^2(D)}^2 + \norm{\nabla \cdot \nabla \cdot \p{M}}_{{L}^2(D)}^2,\\
\norm{ \p{M}}_{\p{H}(D,\textup{symcurl})}^2  &\coloneqq \norm{\p{M}}_{\p{L}^2(D)}^2 + \norm{ \nabla \times  \p{M} + (\nabla \times \p{M})^T}_{\p{L}^2(D)}^2.
\end{alignat*}
\normalsize
Then writing  $\p{C}_c^{\infty}(D)$ for the space of smooth compact supported vector-valued, matrix-valued respectively, functions, one can introduce some spaces with zero boundary  conditions in the sense

\begin{alignat*}{3}
 \p{H}_0^k(D) \coloneqq \overline{\p{C}_c^{\infty}(D)}^{\p{H}^k(D)}, \  { \overset{\circ}{\p{H}}(D,\textup{divdiv}, \mathbb{S})} \coloneqq \overline{\p{C}_c^{\infty}(D) \cap \p{L}^2(D,\mathbb{S})}^{\p{H}(D, \textup{divdiv})}, \  \\  \hspace{2cm} { \overset{\circ}{\p{H}}(D,\textup{symcurl}, \mathbb{T})} \coloneqq \overline{\p{C}_c^{\infty}(D) \cap \p{L}^2(D,\mathbb{T})}^{\p{H}(D, \textup{symcurl})},
\end{alignat*}
where we write $\overline{X}^{Y}$ for the closure of the space $X$ w.r.t. the norm $\norm{\cdot}_Y$.

Next we define some abbreviations. If we have arbitrary functions $t_j, \ s_i \colon D \rightarrow \mathbb{R}, \ i=1,\dots,6, \ j=1, \dots , 8$ and $m_l  \colon D \rightarrow \mathbb{R}, \ l = 1, \dots, 9$, we define the corresponding symmetric and traceless matrix functions through
\small
\begin{align*}
\textup{SYM}(s_1, \dots ,s_6 ) &\colon D  \rightarrow \mathbb{R}^{3 \times 3} , \ \ \textup{SYM}(s_1, \dots ,s_6 ) = \begin{pmatrix}
s_1 & s_2 & s_3 \\
s_2 & s_4 & s_5 \\
s_3 & s_5 & s_6
\end{pmatrix},\\
\textup{TR}(t_1, \dots ,t_8 ) &\colon D  \rightarrow \mathbb{R}^{3 \times 3} , \ \ \textup{TR}(t_1, \dots ,t_8 ) = \begin{pmatrix}
t_1 & t_2 & t_3 \\
t_4 & t_5-t_1 & t_6 \\
t_7 & t_8 & -t_5
\end{pmatrix},\\
\textup{MAT}(m_1, \dots ,m_9 ) &\colon D  \rightarrow \mathbb{R}^{3 \times 3} , \ \ \textup{MAT}(m_1, \dots ,m_9 ) = \begin{pmatrix}
m_1 & m_2 & m_3 \\
m_4 & m_5 & m_6 \\
m_7 & m_8 & m_9
\end{pmatrix}.
\end{align*}
\normalsize
For  operators  $o_i \colon X \rightarrow Y, \ i = 1,\dots , 9$  we can define analogously   the matrix operators, for example \small
$$\textup{MAT}(o_1,\dots,o_9 ) \colon X^{3 \times 3}  \rightarrow  Y^{3 \times 3}, \ \ \textup{MAT}(o_1, \dots ,o_9 )(m_1,\dots,m_9) = \small \begin{pmatrix}
o_1(m_1) & o_2(m_2) & o_3(m_3) \\
o_4(m_4) & o_5(m_5) & o_6(m_6) \\
o_7(m_7) & o_8(m_8) & o_9(m_9)
\end{pmatrix}.$$ \normalsize
Further, for a matrix $\p{M}=(M_{ij})$ we use an upper index $M^j$ to denote the j-th column and a lower index $M_i$ for the $i$-th row. Then, for matrix-valued mapping, we write for the deviatoric gradient \textup{dev}$\nabla$ and the symmetric curl operator $\textup{sym}\nabla \times $. 

From functional analysis we know the Hilbert space adjoint for a densely-defined linear operator $T \colon W^1 \rightarrow W^2, \ \ W^1 , W^2$ Hilbert spaces. It is the  linear mapping $T^* \colon W^2 \rightarrow W^1$ for which $ \langle T x , y \rangle_{W^2} = \langle  x , T^*y \rangle_{W^1}, \ \forall x \in D(T),\ y \in D(T^*) $, where the angle brackets stand for the inner product and $D(L)$ denotes the domain of some operator $L$.\\
After stating some basic notation we proceed  with the consideration of Hilbert complexes, B-splines and spaces involving splines.

\subsection{Hilbert complexes}
The following definitions and explanations are based on the references \cite{ArnoldBook,Pauly2016OnCA}.\\
A \emph{Hilbert complex} is a chain of Hilbert spaces $(W^k,\langle \cdot , \cdot \rangle_{W^k})$ together with closed and densely defined linear  operators
$d^k \colon W^k \rightarrow W^{k+1}$, where one requires  the range  $\mathcal{R}(d^k)$ of $d^k$  to be a subset of the nullspace $\mathcal{N}(d^{k+1})$ of $d^{k+1}$. Hence $d^{k+1} \circ d^k = 0$.
For our purposes we are mainly interested in the so-called  \emph{domain complex} $(V^k,d^k)$, where the Hilbert spaces $W^k$ are replaced by the dense domains $V^k$ of the operators $d^k$, i.e. we have a sequence 
	 \begin{figure}[h]
	\begin{center}
		\begin{tikzpicture}
		\node at (0.2,1.7) {$d^{k-2}$};
		\draw[->] (-0.3,1.4) to (0.8,1.4);
     	\node at (-0.8,1.4) {$\dots$};
		
		\node at (1.5,1.4) {$V^{k-1}$};
		
		\node at (4,1.4) {$V^{k}$};

		\node at (6.5,1.4) {$V^{k+1}$};

		\node at (8.7,1.4) {$\dots \ \ \ .$};

		\node at (2.8,1.7) {$d^{k-1}$};

		\node at (5.1,1.7) {$d^{k} $};

		\node at (7.7,1.7) {$d^{k+1} $};

		\draw[->] (2.2,1.4) to (3.3,1.4);

		\draw[->] (4.5,1.4) to (5.6,1.4);
		
		\draw[->] (7,1.4) to (8.1,1.4);
		\end{tikzpicture}
	\end{center}
\end{figure}

 Using the graph inner product with induced graph norm $ \norm{v }_{V^k}^2 \coloneqq \norm{v}_{W^k}^2+ \norm{d^k v}_{W^{k+1}}^2,$ we obtain with $(V^k,\langle \cdot , \cdot \rangle_{V^k})$ again Hilbert spaces and thus the domain complex is indeed a Hilbert complex. We call a Hilbert complex \emph{closed} if the ranges $\mathcal{R}(d^k) $ are closed in $V^{k+1}$ and we denote the domain complex \emph{exact}, if $\mathcal{R}(d^k)=\mathcal{N}(d^{k+1})$. Another important notion is the \emph{dual complex} which is built up by means of the adjoint operators $d_k^* \coloneqq (d^{k-1})^*$. More precisely, the dual complex of the domain complex has the form 
 
 \hspace{2cm}
		\begin{tikzpicture}
		\node at (0.3,1.7) {$d_{k-1}^{*}$};
		\draw[<-] (-0.3,1.4) to (0.8,1.4);
		\node at (-0.8,1.4) {$\dots$};
		
		\node at (1.5,1.5) {${V_{k-1}^*}$};
		
		\node at (4,1.5) {${V_k^*}$};

		\node at (6.4,1.5) {${V^*_{k+1}}$};

		\node at (8.7,1.4) {$ \ \ \ \ \dots \ \ \ \ ,$};

		\node at (2.8,1.7) {$d_{k}^*$};

		\node at (5.1,1.7) {$d_{k+1}^* $};

		\node at (7.65,1.7) {$d_{k+2}^{*} $};

		\draw[<-] (2.2,1.4) to (3.3,1.4);

		\draw[<-] (4.5,1.4) to (5.6,1.4);
		
		\draw[<-] (7,1.4) to (8.1,1.4);
		\end{tikzpicture}
		\newline where the  $V_k^*$ indicate the domains of the adjoint operators.   \\
Now we tend to the Hessian complex on which we focus in this article. It is the domain complex

\begin{definition}{(Hessian complex)}\\ \label{definition:Hessian}
	\begin{tikzpicture}
	\node at (0,1.4) {$H^2(\Omega)$};
	
	\node at (3,1.4) {$\p{H}(\Omega,\textup{curl}, \mathbb{S})$};

	\node at (6.5,1.4) {$\p{H}(\Omega,\textup{div}, \mathbb{T})$};

	\node at (9.5,1.4) {$\p{L}^2(\Omega)$};

	\node at (1.3,1.7) {$\nabla^2$};

	\node at (4.7,1.7) {$\nabla \times $};

	\node at (8.2,1.7) {$\nabla \cdot $};

	\draw[->] (0.8,1.4) to (1.8,1.4);

	\draw[->] (4.2,1.4) to (5.3,1.4);
	
	\draw[->] (7.7,1.4) to (8.75,1.4);
	\draw[->] (10.2,1.4) to (10.9,1.4);
	\draw[->] (-1.6,1.4) to (-0.72,1.4);
	\node at (10.55,1.7) {$0$};
	\node at (-1.09,1.65) {$\iota_{P_1}$};
	\node at (11.45,1.4) {$\{0\} , $};
	\node at (-2.25,1.4) {$P_1(\Omega)$};
	\end{tikzpicture} 
\end{definition}
\noindent
derived from the Hilbert complex \newline
\noindent
\begin{tikzpicture}
		\node at (0,1.4) {$L^2(\Omega)$};
		
		\node at (3,1.4) {$\p{L}^2(\Omega, \mathbb{S})$};

		\node at (6.5,1.4) {$\p{L}^2(\Omega, \mathbb{T})$};

		\node at (9.5,1.4) {$\p{L}^2(\Omega)$};

		\node at (1.3,1.7) {$\nabla^2$};

		\node at (4.7,1.7) {$\nabla \times $};

		\node at (8.2,1.7) {$\nabla \cdot $};

		\draw[->] (0.8,1.4) to (1.8,1.4);

		\draw[->] (4.2,1.4) to (5.3,1.4);
		
		\draw[->] (7.7,1.4) to (8.75,1.4);
		\draw[->] (10.2,1.4) to (10.9,1.4);
		\draw[->] (-1.6,1.4) to (-0.72,1.4);
		\node at (10.55,1.7) {$0$};
		\node at (-1.09,1.65) {$\iota_{P_1}$};
		\node at (11.45,1.4) {$\{0\} \ .$};
		\node at (-2.25,1.4) {$P_1(\Omega)$};
		\end{tikzpicture} \newline
Above the $\iota_{P_1}$ stands for the inclusion map, $P_1(\Omega)$ denotes the  linear polynomial space and $\Omega \subset \mathbb{R}^3$. Further we make here the assumption:
\begin{assumption}
	 $\Omega$ is a bounded and simply connected  Lipschitz domain with connected boundary.
\end{assumption}
\noindent
And the dual complex has the form \newline
\noindent
\small
		\begin{tikzpicture}
		\node at (0,1.4) {$L^2(\Omega)$};
		
		\node at (2.9,1.5) {${ \overset{\circ}{\p{H}}(\Omega,\textup{divdiv}, \mathbb{S})}$};

		\node at (6.7,1.5) {${\overset{\circ}{\p{H}}(\Omega,\textup{symcurl}, \mathbb{T})}$};

		\node at (9.9,1.4) {${\p{H}}_0^1(\Omega)$};

		\node at (1.2,1.7) {$(\nabla \cdot)^2$};

		\node at (4.7,1.7) { \small $ \textup{sym}\nabla \times $};

		\node at (8.65,1.7) {\small $ -\textup{dev}\nabla $};

		\draw[<-] (0.6,1.4) to (1.6,1.4);

		\draw[<-] (4.3,1.4) to (5.25,1.4);
		
		\draw[<-] (8.3,1.4) to (9.05,1.4);
		\draw[<-] (10.6,1.4) to (11.3,1.4);
		\draw[<-] (-1.6,1.4) to (-0.72,1.4);
		\node at (10.95,1.7) {$0$};
		\node at (-1.09,1.65) {$\pi_{P_1}$};
		\node at (11.75,1.4) {$\{0\} \ ,$};
		\node at (-2.25,1.4) {$P_1(\Omega)$};
		\end{tikzpicture}
\normalsize \noindent \\
 where the circles should indicate that the domains of the dual operators are subspaces of
 \small
 \normalsize
  $\p{L}^2(\Omega)$ with suitable zero boundary conditions.
  Further, we wrote $\pi_{P_1}$ for the $L^2$-orthogonal  projection onto the  linear polynomial space.
  
 One can show that the mentioned Hessian complex is a closed and exact complex. For a proof of the exactness we refer to Theorem 3.3 in \cite{Chen2020DiscreteHC}. And since the exactness also implies that the Hessian sequence is closed ( cf. Theorem 3.8 or  section 4.1 in \cite{ArnoldBook} ) we will be able to set up a discrete Hessian complex based on spline spaces to approximate several PDEs.  To have the mathematical notation available for defining such a finite-dimensional version of the  complex we face some very basic definitions from the field of isogeometric analysis in the next section.

\subsection{Spline spaces}
\label{section:splines}
Here, we state a short overview of B-spline functions, spaces respectively, and some basic results in the univariate as well as in the multivariate case. \\
 Following \cite{IGA1,IGA3} for a brief
exposition, we call an  increasing sequence of real numbers $\Xi \coloneqq \{ \xi_1 \leq  \xi_2  \leq \dots \leq \xi_{n+p+1}  \}$ for some $p \in \mathbb{N}$   \emph{knot vector}, where we assume  $0=\xi_1=\xi_2=\dots=\xi_{p+1}, \ \xi_{n+1}=\xi_{n+2}=\dots=\xi_{n+p+1}=1$, and call such knot vectors $p$-open. 
Furthermore, the multiplicity of the $j$-th knot is denoted by $m(\xi_j)$.
Then  the univariate B-spline functions $\widehat{B}_{j,p}(\cdot)$ of degree $p$ corresponding to a given knot vector $\Xi$ are defined recursively by the \emph{Cox-DeBoor formula}  :
\begin{align*}
\widehat{B}_{j,0}(\zeta) \coloneqq \begin{cases}
1, \ \ \textup{if}  \ \zeta \in [\xi_{j},\xi_{j+1}) \\
0, \ \ \textup{else},
\end{cases}
\end{align*}
\textup{and if }  $p \in \mathbb{N}_{\geq 1} \ \textup{we set}$ 
\begin{align*}
\widehat{B}_{j,p}(\zeta)\coloneqq \frac{\zeta-\xi_{j}}{\xi_{j+p}-\xi_j} \widehat{B}_{j,p-1}(\zeta)  +\frac{\xi_{j+p+1}-\zeta}{\xi_{j+p+1}-\xi_{j+1}} \widehat{B}_{j+1,p-1}(\zeta),
\end{align*}
where one puts $0/0=0$ to obtain  well-definedness. The knot vector $\Xi$ without knot repetitions is denoted by $\{ \psi_1, \dots , \psi_m \}$. \\
 The multivariate extension of the last spline definition is achieved by a tensor product construction. In other words, we set for a given  knot vector   $\boldsymbol{\Xi} \coloneqq \Xi_1 \times   \dots \times \Xi_d $, where the $\Xi_{l}=\{ \xi_1^{l}, \dots , \xi_{n_l+p_l+1}^{l} \}, \ l=1, \dots , d$ are $p_l$-open,   and a given \emph{degree vector}   $\p{p} \coloneqq (p_1, \dots , p_d)$ for the multivariate case
\begin{align*}
\widehat{B}_{\p{i},\p{p}}(\boldsymbol{\zeta}) \coloneqq \prod_{l=1}^{d} \widehat{B}_{i_l,p_l}(\zeta_l), \ \ \ \ \forall \, \p{i} \in \mathit{\mathbf{I}}, \ \  \boldsymbol{\zeta} \coloneqq (\zeta_1, \dots , \zeta_d),
\end{align*}
	with  $d$ as  the underlying dimension of the parametric domain $\widehat{\Omega}= (0,1)^d$ and $\textup{\p{I}}$ the multi-index set $\textup{\p{I}} \coloneqq \{ (i_1,\dots,i_d) \  | \  1\leq i_l \leq n_l, \ l=1,\dots,d  \}$.\\
	 	B-splines  fulfill several properties and for our purposes the most important ones are:
	\begin{itemize}
		\item If  for all internal knots the multiplicity satisfies $1 \leq m(\xi_j) \leq m \leq p , $ then the B-spline basis functions $\widehat{B}_{i,p}({\xi})$ are globally $C^{p-m}$-continuous. Therefore we define  in this case the regularity integer $r \coloneqq p-m$. Obviously, by the product structure, we get splines $\widehat{B}_{\p{i},\p{p}}$ which are $C^{r_l}$-smooth  w.r.t. the $l$-th coordinate direction if the internal multiplicities fulfill $1 \leq m(\xi_j^l) \leq m_l  \leq p_l, \ r_l \coloneqq p_l-r_l, \ \forall l $ in the multivariate case.   We write in the following $\p{r} \coloneqq (r_1,\dots ,r_d), $ for the regularity vector to indicate the smoothness. In case of $r_i <0$ we have discontinuous splines w.r.t. the $i$-th coordinate direction. 
		\item The B-splines $ \{\widehat{B}_{\p{i},\p{p}} \ | \ \ \p{i} \in \p{I} \}$ are linearly independent.
		\item For univariate splines $\widehat{B}_{i,p}, \ p \geq 1$ we have
		\begin{align}
		\label{eq:soline_der}
		\partial_{\zeta} \widehat{B}_{i,p}(\zeta) = \frac{p}{\xi_{i+p}-\xi_i}\widehat{B}_{i,p-1}(\zeta) -  \frac{p}{\xi_{i+p+1}-\xi_{i+1}}\widehat{B}_{i+1,p-1}(\zeta),
		\end{align} 
		with $\widehat{B}_{1,p-1}(\zeta)\coloneqq \widehat{B}_{n+1,p-1}(\zeta) \coloneqq 0$.
		\item  The support of the spline $\widehat{B}_{i,p} $ is a subset of  the interval $[\xi_i,\xi_{i+p+1}]$. Moreover,  the knots $\psi_j$ define a subdivision of the interval $(0,1)$ and for each element $I = (\psi_j,\psi_{j+1})$ we find an $i$ with $(\psi_j,\psi_{j+1})= (\xi_i,\xi_{i+1})$ and write $\tilde{I} \coloneqq  (\xi_{i-p},\xi_{i+p+1})$ for the so-called support extension.
	\end{itemize}
    The space spanned by all univariate splines $\widehat{B}_{i,p}$ corresponding to  given knot vector and degree $p$ and global regularity $r$  is denoted by $$S_p^r \coloneqq \textup{span}\{ \widehat{B}_{i,p} \ | \ i = 1,\dots , n \}.$$
    For the multivariate case we just define the spline space as the product space $$S_{p_1, \dots , p_d}^{r_1,\dots,r_d} \coloneqq S_{p_1}^{r_1} \otimes \dots \otimes S_{p_d}^{r_d} = \textup{span} \{\widehat{B}_{\p{i},\p{p}} \ | \  \p{i} \in \mathit{\mathbf{I}}  \}$$ of proper univariate spline spaces.\\
    To define discrete spaces based on splines we require a parametrization mapping $\p{F} \colon \widehat{\Omega} \coloneqq (0,1)^d \rightarrow \mathbb{R}^d$ which parametrizes the computational domain. In fact we will assume in the subsequent parts that $\p{F}$ is an affine linear map and hence smooth.
    	The knots stored in the knot vector $  \boldsymbol{\Xi} $, corresponding to  the underlying  splines, determine a mesh in the parametric domain $\widehat{\Omega} $, namely  $\widehat{M} \coloneqq \{ K_{\p{j}}\coloneqq (\psi_{j_1}^1,\psi_{j_1+1}^1 ) \times \dots \times (\psi_{j_{d}}^{d},\psi_{j_{d}+1}^{d} ) \ | \  \p{j}=(j_1,\dots,j_{d}), \ \textup{with} \ 1 \leq j_i <m_i\},$ and
    with ${\boldsymbol{\Psi}}= \{\psi_1^1, \dots, \psi _{m_1}^1\}  \times \dots \times \{\psi_1^{d}, \dots, \psi _{m_{d}}^{d}\}$  \  \textup{as  the knot vector} \ ${\boldsymbol{\Xi}}$ \ 
    \textup{without knot repetitions}.   
    The image of this mesh under the mapping $\p{F}$, i.e. $\mathcal{M} \coloneqq \{{\p{F}}(K) \ | \ K \in \widehat{M} \}$, gives us a mesh structure in the physical domain. By inserting knots without changing the parametrization  we can refine the mesh, which is the concept of $h$-refinement; see \cite{IGA2,IGA1,IGA3}.
    For a mesh $\mathcal{M}$ we define the global mesh size $h \coloneqq \max\{h_{\mathcal{K}} \ | \ \mathcal{K} \in \mathcal{M} \}$, where for $\mathcal{K} \in \mathcal{M}$ we denote with $h_{\mathcal{K}} \coloneqq \textup{diam}(\mathcal{K}) $ the \emph{element size}.
    \begin{assumption}{(Regular mesh)}\\
    	There exists a constant $c_u $ independent from the mesh size such that $h_{\mathcal{K}} \leq h \leq c_u \, h_{\mathcal{K}}$ for all mesh elements $\mathcal{K} \in \mathcal{M}$. 
    \end{assumption}
After the introduction of elementary notions we face now the spline-based discretization for the Hessian complex.
\section{Structure-preserving discretization for the Hessian complex}
\label{section_discretization}
Here we consider a Lipschitz domain $\Omega \subset \mathbb{R}^3$ as computational domain for the Hessian complex defined in Section 2.2 (Def. \ref{definition:Hessian}).
 As already mentioned we assume the parametrization mapping to be affine linear, i.e.  $\p{F} \colon (0,1)^3 \rightarrow \Omega  \ , \ \boldsymbol{\zeta} \mapsto \p{A} \cdot \boldsymbol{\zeta} + \p{b} , \ \p{A} \in \mathbb{R}^{3\times 3} \ \textup{invertible}, \ \p{b} \in \mathbb{R}^3 $.
 
 Aim of this section is the establishment of  discrete spaces for the Hessian complex  which fulfill the structure-preserving properties  of the FEEC framework; see  \cite[Section 5.2.2]{ArnoldBook}.\\ 
 To construct an appropriate discretization we follow the approach in \cite{Buffa2011IsogeometricDD} for the case of the de Rham complex. In fact, most of the proofs and results are very similar to the ones in the latter reference, but are  stated for reasons of completeness and since there arise differences in some points. We have to be careful mainly due to the fact that we consider a second-order complex in contrast to the first order sequence in \cite{Buffa2011IsogeometricDD}. \\
 First, we define three mappings which connect the spaces in the reference cube, i.e. in the parametric domain $\widehat{\Omega}$, with the spaces in the physical domain $\Omega$. More precisely, we have the diagram 
	 \begin{figure}[h]
	\begin{center}
		\begin{tikzpicture}
		\node at (0,1.4) {$H^2(\Omega)$};
		
		\node at (3,1.4) {$\p{H}(\Omega,\textup{curl}, \mathbb{S})$};

		\node at (6.5,1.4) {$\p{H}(\Omega,\textup{div}, \mathbb{T})$};

		\node at (9.5,1.4) {$\p{L}^2(\Omega)$};

		\node at (1.3,1.7) {$\nabla^2$};

		\node at (4.7,1.7) {$\nabla \times $};

		\node at (8.2,1.7) {$\nabla \cdot $};

		\draw[->] (0.8,1.4) to (1.8,1.4);

		\draw[->] (4.2,1.4) to (5.3,1.4);
		
		\draw[->] (7.7,1.4) to (8.75,1.4);
		\draw[->] (10.2,1.4) to (10.9,1.4);
		\draw[->] (-1.6,1.4) to (-0.72,1.4);
		\node at (10.55,1.7) {$0$};
		\node at (-1.09,1.65) {$\iota_{P_1}$};
		\node at (11.45,1.4) {$\{0\}$};
		\node at (-2.25,1.4) {$P_1(\Omega)$};

	 \draw[->] (3,1.1) -- (3,0.2);
	  \draw[->] (-2.25,1.1) -- (-2.25,0.2);
	   \draw[->] (11.45,1.1) -- (11.45,0.2);
	    \draw[->] (0,1.1) -- (0,0.2);
	     \draw[->] (6.5,1.1) -- (6.5,0.2);
	      \draw[->] (9.5,1.1) -- (9.5,0.2);

			\node at (0,-0.1) {$H^2(\widehat{\Omega})$};
		
		\node at (3,-0.1) {$\p{H}(\widehat{\Omega},\textup{curl}, \mathbb{S})$};

		\node at (6.5,-0.1) {$\p{H}(\widehat{\Omega},\textup{div}, \mathbb{T})$};

		\node at (9.5,-0.1) {$\p{L}^2(\widehat{\Omega})$};

		\node at (1.3,0.2) {$\widehat{\nabla}^2$};

		\node at (4.7,0.2) {$\widehat{\nabla} \times $};

		\node at (8.2,0.2) {$\widehat{\nabla} \cdot $};

		\draw[->] (0.8,-0.1) to (1.8,-0.1);

		\draw[->] (4.2,-0.1) to (5.3,-0.1);
		
		\draw[->] (7.7,-0.1) to (8.75,-0.1);
		\draw[->] (10.2,-0.1) to (10.9,-0.1);
		\draw[->] (-1.6,-0.1) to (-0.72,-0.1);
		\node at (10.55,0.2) {$0$};
		\node at (-1.09,0.15) {$\iota_{P_1}$};
		\node at (11.57,-0.1) {$\{0\} \ ,$};
		\node at (-2.25,-0.1) {$P_1(\widehat{\Omega})$};

			\node[left] at (-2.2,0.7) { \small $\mathcal{Y}_1$};	
		\node[left] at (0.05,0.7) { \small $\mathcal{Y}_1$};
		\node[left] at (3.05,0.7) { \small $\mathcal{Y}_2$};
		\node[left] at (6.55,0.7) { \small $\mathcal{Y}_3$};
		\node[left] at (9.55,0.7) { \small $\mathcal{Y}_4$};
		\node[left] at (11.45,0.7) { \small $\textup{id}$};
		\end{tikzpicture}
	\end{center}
\end{figure} \\ \noindent
 where 
 \begin{equation}
 \begin{aligned}
 \label{eq:trans_mappings}
 \mathcal{Y}_1  (\phi) &\coloneqq \phi \circ \p{F} ,  \\
  \mathcal{Y}_2  (\p{S}) &\coloneqq  \p{J}^T (\p{S} \circ \p{F}) \ \p{J},  \\
   \mathcal{Y}_3  (\p{T}) &\coloneqq  \textup{det}(\p{J}) \p{J}^T  (\p{T} \circ \p{F})  \p{J}^{-T},  \\
     \mathcal{Y}_4  (\boldsymbol{v}) &\coloneqq  \textup{det}(\p{J}) \p{J}^T  (\boldsymbol{v} \circ \p{F}) ,  
 \end{aligned}
 \end{equation}
with $\widehat{\nabla}$ denoting the nabla operator w.r.t. to the parametric coordinates and $(J_{ij}) \coloneqq \p{J} \coloneqq D\p{F} $ denoting the Jacobian of the parametrization mapping.
The above transformations $\mathcal{Y}_i$ are compatible with the Hessian complex as follows.
 \begin{lemma}
 	\label{lemma:1}
 	The latter diagram  commutes, i.e. $$ \widehat{\nabla}^2 \circ \mathcal{Y}_1 = \mathcal{Y}_2  \circ \nabla^2, \  \ \ \ \ (\widehat{\nabla} \times)  \circ \mathcal{Y}_2= \mathcal{Y}_3 \circ (\nabla \times) , \ \ \ \ \ \ (\widehat{\nabla} \cdot)  \circ \mathcal{Y}_3 = \mathcal{Y}_4  \circ (\nabla \cdot)  \ \ \ \ . $$
 \end{lemma}
\begin{proof}
		We show the three equations separately.
	\begin{itemize}
\item[1.] Let $F_k$ be the $k$-th component function of $\p{F}$. By the chain rule of the hessian operator (see \cite[Corollary 1]{skorski2019chain}) we obtain directly
		        \begin{align*}
		       \widehat{\nabla}^2 (\mathcal{Y}_1(\phi))=\widehat{\nabla}^2(\phi \circ \p{F})  &= \p{J}^T \big((\nabla^2\phi) \circ \p{F}\big)  \p{J} + \sum_k  (\widehat{\nabla}^2F_k)\partial_{k} \phi \\  &=\p{J}^T \big((\nabla^2\phi) \circ \p{F}\big)   \p{J} =  \mathcal{Y}_2 (\nabla^2\phi)   .
		        \end{align*}
              \item[2.] Here we assume that $\p{S}$ is a symmetric $3 \times 3$ matrix-valued function.
		        Then for the second equation in the assertion we first note the  relation\footnote{we use the well-known  generalized product rule $\nabla \times\big(a \ \boldsymbol{v} \big) = \nabla a \  \times \boldsymbol{v} + a \nabla \times \boldsymbol{v}$. }		        		        
		        \begin{align*}
		        \widehat{\nabla}\times \Big[ (a_1,a_2,a_3) \cdot \p{C}(\boldsymbol{\zeta})\Big] = \sum_i \big[  \widehat{\nabla} a_i \times C_i + a_i \ \widehat{\nabla} \times C_i \big],
		        \end{align*}
		        where $(a_1,a_2,a_3) \in \mathbb{R}^3$ is some row vector, $\p{C}(\boldsymbol{\zeta})$ a matrix-valued function and $C_i$ the $i$-th row of the matrix $\p{C}(\boldsymbol{\zeta})$. The last equation can be simplified to      \begin{align}
		        \label{eq:rot1}
		        \widehat{\nabla} \times \Big[ (a_1,a_2,a_3)  \cdot \p{C}(\boldsymbol{\zeta})\Big] = \sum_i \big[ a_i \ \widehat{\nabla} \times C_i \big].
		        \end{align}
		        Using \eqref{eq:rot1} and setting $\p{C} \coloneqq (\p{S} \circ \p{F}) \p{J}$ we can write   
		        \begin{align*}
		        \widehat{\nabla} \times \big[ \p{J}^T (\p{S} \circ \p{F}) \p{J} \big] &=   \widehat{\nabla} \times \Bigg[ \begin{pmatrix}
		        J_{11} & J_{21} & J_{31}\\
		        J_{12} & J_{22} & J_{32}\\
		        J_{13} & J_{23} & J_{33}
		        \end{pmatrix} \p{C} \Bigg]\\
		        &=\begin{pmatrix}
		       \widehat{\nabla} \times (  J^1 \cdot \p{C}) \\
		           \widehat{\nabla} \times (  J^2 \cdot \p{C})\\
		           \widehat{\nabla} \times (  J^3 \cdot \p{C})
		        \end{pmatrix}  =  \begin{pmatrix}
		        \sum_i  J_{i1} \ \widehat{\nabla} \times C_i \\
		     \sum_i  J_{i2} \ \widehat{\nabla} \times C_i\\
		        \sum_i  J_{i3} \ \widehat{\nabla} \times C_i
		        \end{pmatrix}.
		        \end{align*} 
		        Then the next application of the chain rule for the covariant Piola transfomation $\widehat{\nabla} \times (\p{J}^T  (\boldsymbol{v}\circ \p{F})) = \textup{det}(\p{J}) \, \p{J}^{-1}  ( \nabla \times \boldsymbol{v} ) \circ \p{F}$ (see e.g. 2.15 and 2.17 in \cite{Hiptmair2002FiniteEI})  yields\begin{align*}
		        \widehat{\nabla} \times C_i^T& = \widehat{\nabla} \times ((\p{S} \circ \p{F})_i \cdot \p{J})^T = \widehat{\nabla} \times ( \p{J}^T \cdot (\p{S} \circ \p{F})^i)  \\ & =\det(\p{J}) \p{J}^{-1} ((\nabla \times \p{S}^i) \circ \p{F}  ) .
		        \end{align*} This and the symmetry of $\p{S}$ makes it possible to write	        
		        \begin{align*}
		        \widehat{\nabla} \times \big[ \p{J}^T (\p{S} \circ \p{F}) \p{J} \big] & = \begin{pmatrix}
		        \sum_i  J_{i1} \ \det(\p{J})  (({\nabla} \times \p{S}_i) \circ \p{F}  ) \p{J}^{-T} \\
		        \sum_i  J_{i2} \ \det(\p{J})  (({\nabla} \times \p{S}_i) \circ \p{F}  ) \p{J}^{-T}\\
		        \sum_i  J_{i3} \ \det(\p{J})  (({\nabla} \times \p{S}_i) \circ \p{F}  ) \p{J}^{-T}
		        \end{pmatrix}\\
		          & =  \textup{det}(\p{J}) \,  \p{J}^T \ (({\nabla} \times \p{S}) \circ \p{F}  ) \ \p{J}^{-T} \\
		          & = \mathcal{Y}_3(\nabla \times \p{S}).
		        \end{align*}
		        \item[3.] For the last part of the proof we note that for a  vector-valued function  $(a_1,a_2,a_3)^T$ and a matrix-valued function $\p{C}$ it is \begin{align*}
		       \widehat{\nabla} \cdot \big((a_1,a_2,a_3)  \cdot \p{C}(\boldsymbol{\zeta})\big) = \sum_i  C_i \cdot (\widehat{\nabla} a_i)  + a_i \, (\widehat{\nabla} \cdot C_i).
		        \end{align*} This means  with the definition $\p{C} \coloneqq  \textup{det}(\p{J}) \, (\p{M} \circ \p{F})  \p{J}^{-T}$, where $\p{M}$ is a matrix-valued function, we obtain
		        \begin{align*}
		        \widehat{\nabla} \cdot \big(\mathcal{Y}_3(\p{M}) \big)  = \widehat{\nabla} \cdot \big(   \textup{det}(\p{J}) \p{J}^T  (\p{M} \circ \p{F})   \p{J}^{-T}   \big) =  \begin{pmatrix}
		        \sum_i  J_{i1} \ \widehat{\nabla} \cdot  C_i \\
		        \sum_i  J_{i2} \ \widehat{\nabla} \cdot C_i\\
		        \sum_i  J_{i3} \ \widehat{\nabla} \cdot C_i
		        \end{pmatrix}.
		        \end{align*}
		        In view of the divergence preserving Piola transformation $\widehat{\nabla} \cdot \big(\textup{det}(\p{J}) \, \p{J}^{-1}   ( \boldsymbol{v}  \circ \p{F}) \big) = \textup{det}(\p{J}) \, (\nabla \cdot \boldsymbol{v} )\circ \p{F}$ (compare 2.15 and 2.18 in \cite{Hiptmair2002FiniteEI}), it is \begin{align*}
		        \widehat{\nabla}  \cdot C_i &= \widehat{\nabla}  \cdot \big (\textup{det}(\p{J}) (\p{M} \circ \p{F})  \p{J}^{-T})_i =  \widehat{\nabla}  \cdot \big ( \textup{det}(\p{J})  \p{J}^{-1} \ (\p{M}^T \circ \p{F})^i)  \\ 
		        & = \textup{det}(\p{J}) (\nabla \cdot \p{M}_i) \circ \p{F} .
		        \end{align*}
		        And hence one gets 
		        $$ \widehat{\nabla} \cdot \big(\mathcal{Y}_3(\p{M}) \big) = \textup{det}(\p{J}) \, \p{J}^T  \big( (  \nabla \cdot \p{M}) \circ \p{F} \big) = \mathcal{Y}_4(\nabla \cdot \p{M}).  $$
	\end{itemize}
\end{proof}

An important feature of the mappings is the symmetry and trace preservation. We have the next simple lemma.
\begin{lemma}
	The mapping $\mathcal{Y}_2$ is symmetry-preserving and $\mathcal{Y}_3$ preserves the zero trace. 
\end{lemma}
\begin{proof}
	Let $\p{S}$ be a symmetric matrix ,then $$\mathcal{Y}_2(\p{S})^T = \big( \p{J}^T (\p{S} \circ \p{F}) \p{J} \big)^T = \p{J}^T (\p{S} \circ \p{F})^T \p{J} = \p{J}^T (\p{S} \circ \p{F}) \p{J} = \mathcal{Y}_2(\p{S}). $$ Furthermore, since similar matrices have the same traces, the mapping $\mathcal{Y}_3$ maps traceless matrices to traceless matrices. 
More precisely, two matrices $\p{A}, \ \p{B}$ are called similar, if there exists an invertible matrix $\p{C}$ with $\p{A}= \p{C} \p{B} \p{C}^{-1}$.  Then it holds  $\textup{tr}(\p{A})= \textup{tr}(\p{B})$. Consequently we have for $\p{T} \in \p{H}(\Omega,\textup{div}, \mathbb{T})$ the equality chain
$$\textup{tr}(\mathcal{Y}_3(\p{T})) = \textup{tr}( \textup{det}(\p{J}) \p{J}^T  (\p{T} \circ \p{F})   \p{J}^{-T}  ) = \textup{det}(\p{J})  \  \textup{tr}(  \p{J}^T  (\p{T} \circ \p{F})   \p{J}^{-T}  )   = 0.$$
\end{proof}
Before exploiting the above transformations, we come back to  splines and introduce auxiliary spline function spaces on the parametric domain. They will be used later to introduce discrete spaces on the actual domain $\Omega$. 
\begin{definition}{(Discrete spaces in the parametric domain)}\\
	\label{def:sicrete_spaces_param}
	As already defined we denote with $S_{p_1,p_2,p_3}^{r_1,r_2,r_3}$ the scalar-valued  spline space with global  regularity $C^{r_i}$ w.r.t. the $i$-th coordinate. Clearly, this parametric spline space depends on the underlying mesh structure, which we assume to have the global mesh size $h$. 	Here we require the regularity to be greater or equal to one, i.e. $r_i \geq 1$. Then we can define the parametric test function spaces		
	\begin{align*}
	\widehat{V}^0_h &\coloneqq P_1(\widehat{\Omega}) \subset S_{p_1,p_2,p_3}^{r_1,r_2,r_3}, \\
	\widehat{V}_h^1 &\coloneqq S_{p_1,p_2,p_3}^{r_1,r_2,r_3} \subset H^2(\widehat{\Omega}), \nonumber \\
	\widehat{\boldsymbol{V}}_h^2 & \coloneqq \{ \textup{SYM}(s_1, \dots , s_6)  \  \ | \ (s_1, \dots , s_6) \in \\ &( S_{p_1-2,p_2,p_3}^{r_1-2,r_2,r_3} \times S_{p_1-1,p_2-1,p_3}^{r_1-1,r_2-1,r_3} \times S_{p_1-1,p_2,p_3-1}^{r_1-1,r_2,r_3-1} \times S_{p_1,p_2-2,p_3}^{r_1,r_2-2,r_3} \times S_{p_1,p_2-1,p_3-1}^{r_1,r_2-1,r_3-1} \times S_{p_1,p_2,p_3-2}^{r_1,r_2,r_3-2} ) \},	\nonumber \\
	\widehat{\boldsymbol{V}}_h^3 & \coloneqq  \{ \textup{TR}(t_1,\dots ,t_8)  \  \  | \ (t_1,\dots ,t_8) \in   \nonumber \\  & \hspace{1cm} ( S_{p_1-1,p_2-1,p_3-1}^{r_1-1,r_2-1,r_3-1} \times S_{p_1-2,p_2,p_3-1}^{r_1-2,r_2,r_3-1} \times S_{p_1-2,p_2-1,p_3}^{r_1-2,r_2-1,r_3} \times S_{p_1,p_2-2,p_3-1}^{r_1,r_2-2,r_3-1} \\ & \hspace{1.8cm}\times S_{p_1-1,p_2-1,p_3-1}^{r_1-1,r_2-1,r_3-1} \times S_{p_1-1,p_2-2,p_3}^{r_1-1,r_2-2,r_3} \times S_{p_1,p_2-1,p_3-2}^{r_1,r_2-1,r_3-2} \times S_{p_1-1,p_2,p_3-2}^{r_1-1,r_2,r_3-2} ) \}, \nonumber \\
	\widehat{\boldsymbol{V}}_h^4 & \coloneqq  S_{p_1-2,p_2-1,p_3-1}^{r_1-2,r_2-1,r_3-1} \times  S_{p_1-1,p_2-2,p_3-1}^{r_1-1,r_2-2,r_3-1} \times  S_{p_1-1,p_2-1,p_3-2}^{r_1-1,r_2-1,r_3-2}.
	\end{align*}
\end{definition}

\begin{remark}
	Due to the product structure of the splines it is easy to see that the differential operators $\widehat{\nabla}^2, \ \widehat{\nabla} \times$ and $\widehat{\nabla} \cdot $ match with our parametric spline spaces in the sense that $\widehat{\nabla}^2(\widehat{{V}}_h^1) \subset \widehat{\boldsymbol{V}}_h^2 , \ \widehat{\nabla} \times (\widehat{\boldsymbol{V}}_h^2) \subset \widehat{\boldsymbol{V}}_h^3$ and further $\widehat{\nabla} \cdot (\widehat{\boldsymbol{V}}_h^3) \subset \widehat{\boldsymbol{V}}_h^4$. This is an easy consequence of \eqref{eq:soline_der}. Moreover, looking at the regularity property of splines in Section \ref{section:splines}, we have indeed $\widehat{V}_h^1 \in H^2(\widehat{\Omega}),\widehat{\boldsymbol{V}}_h^2 \subset \p{H}(\widehat{\Omega},\textup{curl}, \mathbb{S}) , \ \widehat{\boldsymbol{V}}_h^3 \subset \p{H}(\widehat{\Omega},\textup{div}, \mathbb{T}). $ Besides,	 we want to remark that the space $P_1(\widehat{\Omega})$ is a subspace of every spline space if we have $p_i \geq 1$. Hence we can ignore the discretization step for the first space in the Hessian complex chain; see Def. \ref{definition:Hessian}. 
\end{remark}

In the next step we define the discrete spaces in the physical domain by means of the mappings $\mathcal{Y}_i$; see \eqref{eq:trans_mappings}.

\begin{definition}{(Discrete spaces in the physical  domain)}\\
The test spaces in the physical domain are defined by
 \begin{align*}
 V_h^0 \coloneqq \{ p \ | \ \mathcal{Y}_1(q) \in  P_1(\widehat{\Omega}) \}, \nonumber \\
V_h^1 \coloneqq \{ \phi_h \ | \ \mathcal{Y}_1(\phi_h) \in  \widehat{V}_h^1 \}, \nonumber \\
\boldsymbol{V}_h^2 \coloneqq \{ \p{S}_h \ | \ \mathcal{Y}_2(\p{S}_h ) \in  \widehat{\boldsymbol{V}}_h^2 \},  \\
\boldsymbol{V}_h^3 \coloneqq \{ \p{T}_h \ | \ \mathcal{Y}_3(\p{T}_h ) \in  \widehat{\boldsymbol{V}}_h^3 \}, \nonumber \\
\boldsymbol{V}_h^4 \coloneqq \{ \p{M}_h \ | \ \mathcal{Y}_4(\p{M}_h ) \in  \widehat{\boldsymbol{V}}^4_h \}. \nonumber
\end{align*}
\end{definition}

In particular due to the commutativity property of the mappings $\mathcal{Y}_i$ in  Lemma \ref{lemma:1} we obtain the discrete commuting  diagram in Fig. \ref{fig:discretediagram1}  .

	 \begin{figure}[h]
	\begin{center}
		\begin{tikzpicture}
		\node at (0,1.4) {$V_h^1$};
		
		\node at (3,1.4) {$\boldsymbol{V}^2_h$};

		\node at (6.5,1.4) {$\boldsymbol{V}^3_h$};

		\node at (9.5,1.4) {$\boldsymbol{V}^4_h$};

		\node at (1.3,1.7) {$\nabla^2$};

		\node at (4.7,1.7) {$\nabla \times $};

		\node at (8.2,1.7) {$\nabla \cdot $};

		\draw[->] (0.8,1.4) to (1.8,1.4);

		\draw[->] (4.2,1.4) to (5.3,1.4);
		
		\draw[->] (7.7,1.4) to (8.75,1.4);
		\draw[->] (10.2,1.4) to (10.9,1.4);
		\draw[->] (-1.6,1.4) to (-0.72,1.4);
		\node at (10.55,1.7) {$0$};
		\node at (-1.09,1.65) {$\iota_{P_1}$};
		\node at (11.45,1.4) {$\{0\}$};
		\node at (-2.25,1.4) {$P_1(\Omega)$};

		\draw[->] (3,1.1) -- (3,0.2);
		\draw[->] (-2.25,1.1) -- (-2.25,0.2);
		\draw[->] (11.45,1.1) -- (11.45,0.2);
		\draw[->] (0,1.1) -- (0,0.2);
		\draw[->] (6.5,1.1) -- (6.5,0.2);
		\draw[->] (9.5,1.1) -- (9.5,0.2);

		\node at (0,-0.1) {$\widehat{\boldsymbol{V}}^1_h$};
		
		\node at (3,-0.1) {$\widehat{\boldsymbol{V}}^2_h$};

		\node at (6.5,-0.1) {$\widehat{\boldsymbol{V}}^3_h$};

		\node at (9.5,-0.1) {$\widehat{\boldsymbol{V}}^4_h$};

		\node at (1.3,0.2) {$\widehat{\nabla}^2$};

		\node at (4.7,0.2) {$\widehat{\nabla} \times $};

		\node at (8.2,0.2) {$\widehat{\nabla} \cdot $};

		\draw[->] (0.8,-0.1) to (1.8,-0.1);

		\draw[->] (4.2,-0.1) to (5.3,-0.1);
		
		\draw[->] (7.7,-0.1) to (8.75,-0.1);
		\draw[->] (10.2,-0.1) to (10.9,-0.1);
		\draw[->] (-1.6,-0.1) to (-0.72,-0.1);
		\node at (10.55,0.2) {$0$};
		\node at (-1.09,0.15) {$\iota_{P_1}$};
		\node at (11.45,-0.1) {$\{0\}$};
		\node at (-2.25,-0.1) {$P_1(\widehat{\Omega})$};

		\node[left] at (-2.2,0.7) { \small $\mathcal{Y}_1$};	
		\node[left] at (0.05,0.7) { \small $\mathcal{Y}_1$};
		\node[left] at (3.05,0.7) { \small $\mathcal{Y}_2$};
		\node[left] at (6.55,0.7) { \small $\mathcal{Y}_3$};
		\node[left] at (9.55,0.7) { \small $\mathcal{Y}_4$};
		\node[left] at (11.45,0.7) { \small $\textup{id}$};
		\end{tikzpicture}
	\end{center}
\caption{\small The transformations $\mathcal{Y}_i$ are used to relate the test spaces in parametric and physical domain.}
		\label{fig:discretediagram1}
\end{figure} \normalsize
For the subsequent considerations we use the  abbreviations 
\begin{equation}
\begin{alignedat}{3}
\label{eq:notation_of_the spaces}
V^1 &\coloneqq  H^2(\Omega),  \hspace{2cm} &&\widehat{V}^1 &&\coloneqq H^2(\widehat{\Omega}),\\
\boldsymbol{V}^2 &\coloneqq  \p{H}(\Omega,\textup{curl}, \mathbb{S}),  \hspace{2cm} &&\widehat{\boldsymbol{V}}^2 &&\coloneqq \p{H}(\widehat{\Omega},\textup{curl}, \mathbb{S}),\\
\boldsymbol{V}^3 &\coloneqq  \p{H}(\Omega,\textup{div}, \mathbb{T}),  \hspace{2cm} &&\widehat{\boldsymbol{V}}^3 &&\coloneqq \p{H}(\widehat{\Omega},\textup{div}, \mathbb{T}),\\
\boldsymbol{V}^4 &\coloneqq  \p{L}^2(\Omega),  \hspace{2cm} &&\widehat{\boldsymbol{V}}^4 &&\coloneqq \p{L}^2(\widehat{\Omega}).
\end{alignedat}
\end{equation}
Besides we write $V^0 \coloneqq P_1(\Omega)$ and $\widehat{V}^0 \coloneqq P_1(\widehat{\Omega})$.

Aim for the next part is to  show the existence of bounded projection operators $\Pi^1_h \colon V^1 \rightarrow V_h^1,$ $\ \Pi^i_h \colon \boldsymbol{V}^i \rightarrow \boldsymbol{V}_h^i, \ i>1 $, which satisfy suitable approximation estimates.
As the approach in \cite{Buffa2011IsogeometricDD}, we first look at the parametric case, i.e. we define projections of the form
$$\widehat{\Pi}^1_h \colon \widehat{V}^1 \rightarrow \widehat{V}_h^1,\ \ \ \widehat{\Pi}^i_h \colon \widehat{\boldsymbol{V}}^i \rightarrow \widehat{\boldsymbol{V}}_h^i.$$
 Underlying for the projection construction are projections onto spaces of univariate splines. Hence the approach exploits the product structure of our spline spaces.	We define three different projections, where two of them are the same as in the paper \cite[Section 3.1.2]{Buffa2011IsogeometricDD}. These projection operators are the basis for the projection definitions in the multivariate case.
\begin{definition}
	For $p_1-1 \geq r_1 \geq -1, \ p_2-1 \geq r_2 \geq 0$ and $p_3-1 \geq r_3 \geq 1$ we set
	\begin{align}
	\widehat{\Pi}_{p_1} &\colon L^2((0,1)) \rightarrow S_{p_1}^{r_1}, \  v \mapsto  \sum_i (\lambda_i^{p_1}v)\widehat{B}_{i,p_1},  \ \ \textup{analogous to }(3.6)  \ \textup{in \cite{Buffa2011IsogeometricDD}} \nonumber, \\
	\widehat{\Pi}_{p_2-1}^{c,1} &\colon L^2((0,1)) \rightarrow S_{p_2-1}^{r_2-1}, \ v \mapsto \frac{d }{d \zeta} \widehat{\Pi}_{p_2} \int_{0}^{\zeta} v(\tau) \ \textup{d} \tau,  \ \ \textup{analogous to } (3.9) \ \textup{in \cite{Buffa2011IsogeometricDD}}, \nonumber \\
	\widehat{\Pi}_{p_3-2}^{c,2} &\colon L^2((0,1)) \rightarrow S_{p_3-2}^{r_3-2}, \ v \mapsto \frac{d }{d \zeta} \widehat{\Pi}_{p_3-1}^{c,1} \int_{0}^{\zeta} v(\tau) \ \textup{d} \tau  = \frac{d^2}{d \zeta^2} \widehat{\Pi}_{p_3} \int_{0}^{\zeta} \int_{0}^{\tau} v(r) \ \textup{d}r\, \textup{d} \tau. \nonumber	
	\end{align}
	Above $\lambda_i^p$ denote the canonical dual basis functionals corresponding to the spline basis functions, i.e. $\lambda_i^p( \widehat{B}_{j,p}) = \delta_{ij}$. For more information we refer to \cite[Section 4.6]{schumaker_2007}.
\end{definition}
Useful for our purposes are the subsequent properties.
\begin{lemma}
	\label{lemma:univariate_spline_proj_prop}
	The above  projections of the univariate case satisfy the following  properties, where we assume $r \geq -1, \ p \geq 0$ for the first two lines and $p -1 \geq r \geq 0$ otherwise and let $I \coloneqq (\psi_i,\psi_{i+1})$ denote an arbitrary sub-interval induced by the discretization.
	\begin{align}
	\widehat{\Pi}_{p}s &= s  & \forall s \in S^{r}_{p} , \label{eq:lemma_uni_spline_prop_0}\\		
	|\widehat{\Pi}_{p} v|_{H^l(I)} &\leq C \ 	|v|_{H^l(\tilde{I})}  & \forall v \in H^l((0,1)), \ 0 \leq l \leq p+1, \nonumber \\
		\label{eq:lemma_uni_spline_prop_1}
	\widehat{\Pi}_{p-1}^{c,1}s &= s  & \forall s \in S^{r-1}_{p-1} , \\
		\label{eq:lemma_uni_spline_prop_2}
	\widehat{\Pi}_{p-1}^{c,1} \partial_{\zeta} v &= \partial_{\zeta} \widehat{\Pi}_{p}  v  & \forall v \in H^1((0,1)), \\
		\label{eq:lemma_uni_spline_prop_3}
	|\widehat{\Pi}_{p-1}^{c,1} v|_{H^l(I)} &\leq C \ 	| v|_{H^l(\tilde{I})}  & \forall v \in H^l((0,1)), \ 0 \leq l \leq p,
	\end{align}
	for some constant $C$ independent of mesh refinement.\\
	And if it is $r \geq 1, \ p \geq 2$ we  further have 
	\begin{align}
	\label{eq:lemma_uni_spline_prop_4}
\widehat{\Pi}_{p-2}^{c,2}s &= s  & \forall s \in S^{r-2}_{p-2} , \\
	\label{eq:lemma_uni_spline_prop_5}
\widehat{\Pi}_{p-2}^{c,2} \partial_{\zeta}^2 v &= \partial_{\zeta}^2 \widehat{\Pi}_{p}  v = \partial_{\zeta} \widehat{\Pi}_{p-1}^{c,1}  \partial_{\zeta}v & \forall v \in H^2((0,1)), \\	
|\widehat{\Pi}_{p-2}^{c,2} v|_{H^l(I)} &\leq C \ 	| v|_{H^l(\tilde{I})}  & \forall v \in H^l((0,1)), \ 0 \leq l \leq p-1. \label{eq:lemma_uni_spline_prop_6}
\end{align}
\end{lemma}
\begin{proof}
	 The first five statements correspond to the properties (3.7),\ (3.8) and (3.10)-(3.12) in \cite{Buffa2011IsogeometricDD} if we keep the regular mesh assumption in mind.
	 Hence we only check the points \eqref{eq:lemma_uni_spline_prop_4}-\eqref{eq:lemma_uni_spline_prop_6}.\\
	  In the following we use (3.4) of \cite{Buffa2011IsogeometricDD}, which gives immediately that $\partial_{\zeta}^2 S_{p}^{r} = S_{p-2}^{r-2}$. Firstly, with property \eqref{eq:lemma_uni_spline_prop_1} we have
	 \begin{align*}
	 \widehat{\Pi}_{p-2}^{c,2}s =  \frac{d }{d \zeta} \widehat{\Pi}_{p-1}^{c,1} \underbrace{\int_{0}^{\zeta} s(\tau) \ \textup{d} \tau}_{\in S_{p-1}^{r-1}} =  \frac{d }{d \zeta}  {\int_{0}^{\zeta} s(\tau) \ \textup{d} \tau} =  s.
	 \end{align*}
	 In view of \eqref{eq:lemma_uni_spline_prop_0} and \eqref{eq:lemma_uni_spline_prop_2} one gets 
	 \begin{align*}
	 \widehat{\Pi}_{p-2}^{c,2} \partial_{\zeta}^2 v &= \frac{d^2}{d \zeta^2} \widehat{\Pi}_p \int_{0}^{\zeta} \int_{0}^{\tau}  \partial_{r}^2 v(r) \ \textup{d}r\, \textup{d} \tau =  \frac{d^2}{d \zeta^2} \widehat{\Pi}_p \big(v(\zeta) + a\cdot \zeta + b \big) =  \frac{d^2}{d \zeta^2} \widehat{\Pi}_p \big(v(\zeta)  \big), \\
	 \widehat{\Pi}_{p-2}^{c,2} \partial_{\zeta}^2 v &= \frac{d^2}{d \zeta^2} \widehat{\Pi}_p \int_{0}^{\zeta}  \partial_{r} v(r) + a \ \textup{d}r\, = \frac{d^2}{d \zeta^2} \widehat{\Pi}_p \int_{0}^{\zeta}  \partial_{r} v(r) \ \textup{d}r = \frac{d}{d \zeta} \widehat{\Pi}_{p-1}^{c,1}\partial_{\zeta} v.
	 \end{align*}
	 In the last two lines we wrote $a, b \in \mathbb{R}$ for suitable constants of integration. Hence \eqref{eq:lemma_uni_spline_prop_5} follows.\\ 
	 The last point can be proven in a similar fashion like equation 3.12 in \cite{Buffa2011IsogeometricDD} using the fact that if $v \in H^l((0,1)) $ then $w(\zeta)  \coloneqq \int_{0}^{\zeta} v \textup{d} \tau  \in H^{l+1}((0,1))$ and using estimate \eqref{eq:lemma_uni_spline_prop_3}. It is
	 $$|\widehat{\Pi}_{p-2}^{c,2} v|_{H^l(I)} \leq |\widehat{\Pi}_{p-1}^{c,1} w|_{H^{l+1}(I)} \leq C \, |w|_{H^{l+1}(\tilde{I})}  \leq C \, |v|_{H^{l}( \tilde{I})} .$$
\end{proof}

For the multivariate case we define projections as the tensor product $\otimes$ of the univariate case operators, e.g. $\widehat{\Pi}_{p_1,p_2,p_3} \colon L^2((0,1)^3) \rightarrow S_{p_1,p_2,p_3}^{r_1,r_2,r_3}, \ \ v \mapsto (\widehat{\Pi}_{p_1} \otimes \widehat{\Pi}_{p_2} \otimes \widehat{\Pi}_{p_3})v$.
Here $\widehat{\Pi}_{p_i}$ acts in some sense only in the $i$-th coordinate. For more information concerning the interpretation of  tensor product projections we refer to Section 4.1. in \cite{Buffa2011IsogeometricDD}.

Next we can define the projections for the parametric Hessian complex, namely we set for $p_i \geq 2 $ and $r_i \geq 1$.

\begin{definition}{(Projections in the parametric domain)}
	\small
	\begin{align*}
	\widehat{\Pi}^1_h & \coloneqq \widehat{\Pi}_{p_1} \otimes \widehat{\Pi}_{p_2} \otimes \widehat{\Pi}_{p_3}, \\
	\widehat{{\Pi}}^2_h &\coloneqq \textup{SYM} \Big( \big(\widehat{\Pi}_{p_1-2}^{c,2} \otimes \widehat{\Pi}_{p_2} \otimes \widehat{\Pi}_{p_3} \big)   \times  \big(\widehat{\Pi}_{p_1-1}^{c,1} \otimes \widehat{\Pi}_{p_2-1}^{c,1} \otimes \widehat{\Pi}_{p_3} \big) \times \big(\widehat{\Pi}_{p_1-1}^{c,1} \otimes \widehat{\Pi}_{p_2} \otimes \widehat{\Pi}_{p_3-1}^{c,1} \big) \nonumber \\ &  \hspace{1cm} \times \big(\widehat{\Pi}_{p_1} \otimes \widehat{\Pi}_{p_2-2}^{c,2} \otimes \widehat{\Pi}_{p_3} \big) \times \big(\widehat{\Pi}_{p_1}\otimes \widehat{\Pi}_{p_2-1}^{c,1} \otimes \widehat{\Pi}_{p_3-1}^{c,1} \big) \times \big(\widehat{\Pi}_{p_1} \otimes \widehat{\Pi}_{p_2} \otimes \widehat{\Pi}_{p_3-2}^{c,2} \big) \Big), \\
	\widehat{{\Pi}}^3_h & \coloneqq \textup{MAT} \Big( \big(\widehat{\Pi}_{p_1-1}^{c,1} \otimes \widehat{\Pi}_{p_2-1}^{c,1} \otimes \widehat{\Pi}_{p_3-1}^{c,1} \big) \times \big(\widehat{\Pi}_{p_1-2}^{c,2} \otimes \widehat{\Pi}_{p_2} \otimes \widehat{\Pi}_{p_3-1}^{c,1} \big) \times \big(\widehat{\Pi}_{p_1-2}^{c,2} \otimes \widehat{\Pi}_{p_2-1}^{c,1} \otimes \widehat{\Pi}_{p_3} \big) \nonumber \\ & \hspace{1cm} \times \big(\widehat{\Pi}_{p_1} \otimes \widehat{\Pi}_{p_2-2}^{c,2} \otimes \widehat{\Pi}_{p_3-1}^{c,1} \big) \times \big(\widehat{\Pi}_{p_1-1}^{c,1} \otimes \widehat{\Pi}_{p_2-1}^{c,1} \otimes \widehat{\Pi}_{p_3-1}^{c,1} \big) \times \big(\widehat{\Pi}_{p_1-1}^{c,1} \otimes \widehat{\Pi}_{p_2-2}^{c,2} \otimes \widehat{\Pi}_{p_3} \big) \nonumber \\
	& \hspace{1cm} \times  \big(\widehat{\Pi}_{p_1} \otimes \widehat{\Pi}_{p_2-1}^{c,1} \otimes \widehat{\Pi}_{p_3-2}^{c,2} \big) \times  \big(\widehat{\Pi}_{p_1-1}^{c,1} \otimes \widehat{\Pi}_{p_2} \otimes \widehat{\Pi}_{p_3-2}^{c,2} \big) \times \big(\widehat{\Pi}_{p_1-1}^{c,1} \otimes \widehat{\Pi}_{p_2-1}^{c,1} \otimes \widehat{\Pi}_{p_3-1}^{c,1} \big) \Big),\\
	\widehat{{\Pi}}^4_h &\coloneqq  \Big( \big(\widehat{\Pi}_{p_1-2}^{c,2} \otimes \widehat{\Pi}_{p_2-1}^{c,1} \otimes \widehat{\Pi}_{p_3-1}^{c,1} \big)   \times  \big(\widehat{\Pi}_{p_1-1}^{c,1} \otimes \widehat{\Pi}_{p_2-2}^{c,2} \otimes \widehat{\Pi}_{p_3-1}^{c,1} \big) \times \big(\widehat{\Pi}_{p_1-1}^{c,1} \otimes \widehat{\Pi}_{p_2-1}^{c,1} \otimes \widehat{\Pi}_{p_3-2}^{c,2} \big) \Big) .
	\end{align*}
\end{definition}
\normalsize

The latter projection operators preserve the structure of the Hessian complex since they commute with the respective differential operators. We have the subsequent lemma.
\begin{lemma}
	It holds \begin{align}
	\label{eq:commutation_param_diff_1}
	\widehat{\nabla}^2 \big(\widehat{{\Pi}}^1_h \widehat{\phi}\big) &= \widehat{{\Pi}}^2_h \big(\widehat{\nabla}^2 \widehat{\phi} \big), &\forall \, \widehat{\phi} \in \widehat{V}^1, \\
	\widehat{\nabla} \times \big(\widehat{{\Pi}}^2_h \widehat{\boldsymbol{S}} \big) &= \widehat{{\Pi}}^3_h \big(\widehat{\nabla} \times \widehat{\boldsymbol{S}} \big), &\forall \, \widehat{\boldsymbol{S}} \in \widehat{\boldsymbol{V}}^2, \nonumber \\
	\widehat{\nabla} \cdot \big(\widehat{{\Pi}}^3_h \widehat{\boldsymbol{T}} \big) &= \widehat{{\Pi}}^4_h \big(\widehat{\nabla} \cdot \widehat{\boldsymbol{T}} \big), &\forall \, \widehat{\boldsymbol{T}} \in \widehat{\boldsymbol{V}}^3. \nonumber
	\end{align}
\end{lemma}
\begin{proof}
	The proof is analogous to the one for Lemma 4.3 in \cite{Buffa2011IsogeometricDD} and a  consequence of the commutation properties \eqref{eq:lemma_uni_spline_prop_2} and \eqref{eq:lemma_uni_spline_prop_5} and the tensor product construction of the splines spaces and projections. For reasons of explanations we show the assertion for two different entries of the matrix  \eqref{eq:commutation_param_diff_1}. The rest follows with very similar ideas. Let $\widehat{\phi}$ be a smooth function with compact support in $\widehat{\Omega}$. \\
	For example we have 
	\begin{align*}
	\Big( 	\widehat{\nabla}^2 \big(\widehat{{\Pi}}^1_h \widehat{\phi}\big)  \Big)_{11} &= \partial_{\zeta_1}^2 \big(\widehat{{\Pi}}^1_h \widehat{\phi}\big) =   \partial_{\zeta_1}^2 \big( \widehat{\Pi}_{p_1} \otimes \widehat{\Pi}_{p_2} \otimes \widehat{\Pi}_{p_3} \big)\widehat{\phi} =\partial_{\zeta_1}^2 ( \widehat{\Pi}_{p_1} ( \widehat{\Pi}_{p_2} ( \widehat{\Pi}_{p_3} \widehat{\phi})))= \\
	&=  ( \widehat{\Pi}_{p_1-2}^{c,2} ( \widehat{\Pi}_{p_2} ( \widehat{\Pi}_{p_3} \partial_{\zeta_1}^2\widehat{\phi}))) = \big(\widehat{\Pi}_{p_1-2}^{c,2} \otimes \widehat{\Pi}_{p_2} \otimes \widehat{\Pi}_{p_3}\big) \partial_{\zeta_1}^2\widehat{\phi} = \Big( \widehat{{\Pi}}^2_h 	\widehat{\nabla}^2 \big( \widehat{\phi}\big)  \Big)_{11}.
	\end{align*}
	On the other hand we get e.g.
	\begin{align*}
		\Big( 	\widehat{\nabla}^2 \big(\widehat{{\Pi}}^1_h \widehat{\phi}\big)  \Big)_{12} &= \partial_{\zeta_1} \partial_{\zeta_2} \big(\widehat{{\Pi}}^1_h \widehat{\phi}\big) = \partial_{\zeta_1} \partial_{\zeta_2} ( \widehat{\Pi}_{p_1} ( \widehat{\Pi}_{p_2} ( \widehat{\Pi}_{p_3} \widehat{\phi}))) = \partial_{\zeta_1} ( \widehat{\Pi}_{p_1} ( \widehat{\Pi}_{p_2-1}^{c,1} ( \widehat{\Pi}_{p_3} \partial_{\zeta_2}\widehat{\phi}))) \\
	&=   ( \widehat{\Pi}_{p_1-1}^{c,1} ( \widehat{\Pi}_{p_2-1}^{c,1} ( \widehat{\Pi}_{p_3} \partial_{\zeta_1}\partial_{\zeta_2}\widehat{\phi})))  = \big(\widehat{\Pi}_{p_1-1}^{c,1} \otimes \widehat{\Pi}_{p_2-1}^{c,1} \otimes \widehat{\Pi}_{p_3}\big) \partial_{\zeta_1}\partial_{\zeta_2}\widehat{\phi} \\
	&= \Big( \widehat{{\Pi}}^2_h 	\widehat{\nabla}^2 \big( \widehat{\phi}\big)  \Big)_{12}.
	\end{align*}
Doing similar computations for the other components  together with the boundedness of the projections in the univariate case, Lemma  \ref{lemma:stability_projections_param} respectively, and a density argument leads to \eqref{eq:commutation_param_diff_1}.\\
The proof of the rest of the statement uses analogous steps and is not shown here.
\end{proof}

Now the projections in the physical domain are  defined through the mappings $\mathcal{Y}_i$. We set 
\begin{align*}
\Pi^1_h &\coloneqq \mathcal{Y}_1^{-1} \circ \widehat{{\Pi}}^1_h \circ  \mathcal{Y}_1, \\
\Pi^2_h &\coloneqq \mathcal{Y}_2^{-1} \circ \widehat{{\Pi}}^2_h \circ  \mathcal{Y}_2, \\
\Pi^3_h &\coloneqq \mathcal{Y}_3^{-1} \circ \widehat{{\Pi}}^3_h \circ  \mathcal{Y}_3, \\
\Pi^4_h &\coloneqq \mathcal{Y}_4^{-1} \circ \widehat{{\Pi}}^4_h \circ  \mathcal{Y}_4.  
\end{align*}
One notes the invertibility  of the $\mathcal{Y}_i$ and further, because of the fact  $P_1(\widehat{\Omega}) \subset S_{p_1,p_2,p_3}^{r_1,r_2,r_3}  $, we can set $\Pi_h^0 \coloneqq \textup{id}$.
Due to the definition of the latter projections combined with the commutativity properties of the parametric projections as well as of the mappings $\mathcal{Y}_i$, the next diagram commutes.

	 \begin{figure}[h]
	\begin{center}
		\begin{tikzpicture}
		\node at (0,1.4) {$V^1$};
		
		\node at (3,1.4) {$\boldsymbol{V}^2$};

		\node at (6.5,1.4) {$\boldsymbol{V}^3$};

		\node at (9.5,1.4) {$\boldsymbol{V}^4$};

		\node at (1.3,1.7) {$\nabla^2$};

		\node at (4.7,1.7) {$\nabla \times $};

		\node at (8.2,1.7) {$\nabla \cdot $};

		\draw[->] (0.8,1.4) to (1.8,1.4);

		\draw[->] (4.2,1.4) to (5.3,1.4);
		
		\draw[->] (7.7,1.4) to (8.75,1.4);
		\draw[->] (10.2,1.4) to (10.9,1.4);
		\draw[->] (-1.6,1.4) to (-0.72,1.4);
		\node at (10.55,1.7) {$0$};
		\node at (-1.09,1.65) {$\iota_{P_1}$};
		\node at (11.45,1.4) {$\{0\}$};
		\node at (-2.25,1.4) {$P_1(\Omega)$};

		\draw[->] (3,1.1) -- (3,0.2);
		\draw[->] (-2.25,1.1) -- (-2.25,0.2);
		\draw[->] (11.45,1.1) -- (11.45,0.2);
		\draw[->] (0,1.1) -- (0,0.2);
		\draw[->] (6.5,1.1) -- (6.5,0.2);
		\draw[->] (9.5,1.1) -- (9.5,0.2);

		\node at (0,-0.1) {${{V}}^1_h$};
		
		\node at (3,-0.1) {${\boldsymbol{V}}^2_h$};

		\node at (6.5,-0.1) {${\boldsymbol{V}}^3_h$};

		\node at (9.5,-0.1) {${\boldsymbol{V}}^4_h$};

		\node at (1.3,0.2) {${\nabla}^2$};

		\node at (4.7,0.2) {${\nabla} \times $};

		\node at (8.2,0.2) {${\nabla} \cdot $};

		\draw[->] (0.8,-0.1) to (1.8,-0.1);

		\draw[->] (4.2,-0.1) to (5.3,-0.1);
		
		\draw[->] (7.7,-0.1) to (8.75,-0.1);
		\draw[->] (10.2,-0.1) to (10.9,-0.1);
		\draw[->] (-1.6,-0.1) to (-0.72,-0.1);
		\node at (10.55,0.2) {$0$};
		\node at (-1.09,0.15) {$\iota_{P_1}$};
		\node at (11.45,-0.1) {$\{0\}$};
		\node at (-2.25,-0.1) {$V_h^0$};

		\node[left] at (-2.2,0.7) { \small $\textup{id}$};	
		\node[left] at (0.05,0.7) { \small $\Pi_h^1$};
		\node[left] at (3.05,0.7) { \small $\Pi^2_h$};
		\node[left] at (6.55,0.7) { \small $\Pi^3_h$};
		\node[left] at (9.55,0.7) { \small $\Pi^4_h$};
		\node[left] at (11.45,0.7) { \small $\textup{id}$};
		\end{tikzpicture}
	\end{center}
\caption{\small The  diagram commutes and thus the discretization is compatible with the complex.} \label{fig:commuting_digram_3}
\end{figure}
\normalsize
The commutativity of latter diagram makes it possible to proof  later approximation estimates for different model problems. Prior to that we look at approximation properties of the discrete spaces from above. This is subject of the following section.
\section{Approximation estimates}
In the previous section we showed how one can use splines to construct in some sense structure-preserving discrete spaces for the Hessian complex, i.e. the diagram in Fig. \ref{fig:commuting_digram_3} commutes. Hence it is reasonable to face the approximation behavior of the discrete spaces from above in order to set up  convergence estimates for different PDE problems later. As in \cite{Buffa2011IsogeometricDD} we do this by relating the approximation behavior in the parametric domain to the spline based spaces on the physical domain. Underlying for this approach are the next three lemmas which correspond to the Lemma 5.2 , \ Lemma 4.2 and Lemma 5.1 in \cite{Buffa2011IsogeometricDD}, where we consider the case without boundary conditions. Here and in the following we assume for reasons of simplification $p_1=p_2=p_3 = p \geq 2$ and $r_1=r_2=r_3 =r \geq 1$. Further, we introduce the \emph{broken}-Sobolev spaces $\mathcal{H}^{l}$ by
\begin{align*}
\mathcal{H}^{l}(\widehat{\Omega}) \coloneqq  \{ v \in L^2(\widehat{\Omega}) \ | \ v_{| K} \in H^l(K), \ \forall K \in \widehat{M}  \},
\end{align*}
where the corresponding semi-norm is defined via 
\begin{align*}
|v|_{\mathcal{H}^l(\widehat{Q})}^2 \coloneqq \sum_{\underset{K \cap \widehat{Q} \neq \emptyset}{K \in \widehat{{M}}}}  |v|_{H^l(K)}^2,
\end{align*}
with obvious component-wise generalization to  vector- or matrix-valued mappings. In latter case we use again a bold-type notation, i.e. $ \boldsymbol{\mathcal{H}}^l$.
\begin{lemma}{(Regularity preservation)}
	\label{lemma:regularity_transformation}\\
	Due to the fact that $\p{F}$ is smooth we have for some constant $C$, independent of mesh-refinement and of the mesh element $ \mathcal{K}= \p{F}(K)$, $K \in \widehat{{M}}$,
	\begin{alignat*}{4}
	\frac{1}{C} \norm{\phi}_{H^l(\mathcal{K})} &\leq \norm{\mathcal{Y}_1(\phi)}_{H^l(K)} &&\leq C \norm{\phi}_{H^l(\mathcal{K})}, \ \ \ &&\forall \phi \in H^l(\mathcal{K}), \\
		\frac{1}{C} \norm{\p{S}}_{\p{H}^l(\mathcal{K})} &\leq \norm{\mathcal{Y}_2(\p{S})}_{\p{H}^l(K)} &&\leq C \norm{\p{S}}_{\p{H}^l(\mathcal{K})}, \ \ \ &&\forall \p{S} \in \p{H}^l(\mathcal{K}), \\
				\frac{1}{C} \norm{\p{T}}_{\p{H}^l(\mathcal{K})} &\leq \norm{\mathcal{Y}_3(\p{T})}_{\p{H}^l(K)} &&\leq C \norm{\p{T}}_{\p{H}^l(\mathcal{K})}, \ \ \ &&\forall \p{T} \in \p{H}^l(\mathcal{K}),\\
		\frac{1}{C} \norm{\boldsymbol{v}}_{\p{H}^l(\mathcal{K})} &\leq \norm{\mathcal{Y}_4(\boldsymbol{v})}_{\p{H}^l(K)} &&\leq C \norm{\boldsymbol{v}}_{\p{H}^l(\mathcal{K})}, \ \ \ &&\forall \boldsymbol{v} \in \p{H}^l(\mathcal{K}).
	\end{alignat*}
	\begin{proof}
		Follows directly by the smoothness of $\p{F}$.
	\end{proof}
\end{lemma}
\begin{lemma}{(Stability of the projections)}\\
	\label{lemma:stability_projections_param}
	The projections $\widehat{{\Pi}}^1_h, \dots , \widehat{{\Pi}}^4_h$ are continuous in the sense
		\begin{alignat*}{3}
\norm{ \widehat{{\Pi}}^1_h\widehat{\phi}}_{L^2(K)} &&\leq C \norm{\widehat{\phi}}_{L^2(\tilde{K})}, \ \ \ &&\forall \widehat{\phi} \in L^2(\widehat{\Omega}), \\
	 \norm{\widehat{{\Pi}}^2_h \widehat{\boldsymbol{S}}}_{\p{L}^2(K)} &&\leq C \norm{\widehat{\boldsymbol{S}}}_{\p{L}^2(\tilde{K})}, \ \ \ &&\forall \widehat{\boldsymbol{S}} \in \p{L}^2(\widehat{\Omega}), \\
 \norm{\widehat{{\Pi}}^3_h\widehat{\boldsymbol{T}}}_{\p{L}^2(K)} &&\leq C \norm{\widehat{\boldsymbol{T}}}_{\p{L}^2(\tilde{K})}, \ \ \ &&\forall \widehat{\boldsymbol{T}} \in \p{L}^2(\widehat{\Omega}),\\
 \norm{\widehat{{\Pi}}^4_h \widehat{\boldsymbol{v}}}_{\p{L}^2(K)} &&\leq C \norm{\widehat{\boldsymbol{v}}}_{\p{L}^2(\tilde{K})}, \ \ \ &&\forall \widehat{\boldsymbol{v}} \in \p{L}^2(\widehat{\Omega}),
	\end{alignat*}
	for a suitable constant $C = C(\p{F})$. Above $\tilde{K}$ stands for the extended support of the mesh element $K$. This means if $K = (\psi_{i_1}^1,\psi_{i_1+1}^1) \times \dots \times (\psi_{i_3}^3,\psi_{i_3+1}^3)$, we find indices $j_l$ with $K = (\xi_{j_1}^1,\xi_{j_1+1}^1) \times \dots \times (\xi_{j_3}^3,\xi_{j_3+1}^3)$ . Then we set $$\tilde{K} \coloneqq (\xi_{j_1-p_1}^1,\xi_{j_1+p_1+1}^1) \times \dots \times (\xi_{j_3-p_3}^3,\xi_{j_3+p_3+1}^3). $$
	
\end{lemma}
\begin{proof}
	The statement follows by the regular mesh assumption, the product structure of the projections onto the multivariate spline spaces and  Lemma  \ref{lemma:univariate_spline_proj_prop}.
\end{proof}
\begin{lemma}{(Approximation property in the parametric domain)}\\
	\label{lemma:approximation_parametric_domain}
Let $p> r \geq 1$. It  holds
\begin{alignat*}{3}
&|\widehat{\phi}-\widehat{{\Pi}}^1_h\widehat{\phi}|_{H^{l}(K)} \leq C \ h^{s-l} \ |\widehat{\phi}|_{\mathcal{H}^s(\tilde{K})} \ , &&0 \leq l \leq s \leq  p+1 , \ &&\forall \widehat{\phi}  \in \widehat{V}^1 \cap \mathcal{Y}_1(H^s(\Omega)),\\
&|\widehat{\boldsymbol{S}}-\widehat{{\Pi}}^2_h\widehat{\boldsymbol{S}}|_{\p{H}^{l}(K)} \leq C \ h^{s-l} \ |\widehat{\boldsymbol{S}}|_{\boldsymbol{\mathcal{H}}^s(\tilde{K})} \ , \hspace{1cm} &&0 \leq l \leq s \leq  p-1 , \ &&\forall \widehat{\boldsymbol{S}}  \in \widehat{\boldsymbol{V}}^2 \cap \mathcal{Y}_2(\p{H}^s(\Omega)),\\
&|\widehat{\boldsymbol{T}}-\widehat{\Pi}^{3}_h\widehat{\boldsymbol{T}}|_{\p{H}^{l}(K)} \leq C \ h^{s-l} \ |\widehat{\boldsymbol{T}}|_{\boldsymbol{\mathcal{H}}^s(\tilde{K})} \ , &&0 \leq l \leq s \leq  p-1 , \ &&\forall \widehat{\boldsymbol{T}}  \in \widehat{\boldsymbol{V}}^3 \cap \mathcal{Y}_3(\p{H}^s(\Omega)),\\
& |\widehat{\boldsymbol{v}}-\widehat{{\Pi}}^4_h\widehat{\boldsymbol{v}}|_{\p{H}^{l}(K)} \leq C \ h^{s-l} \ |\widehat{\boldsymbol{v}}|_{\boldsymbol{\mathcal{H}}^s(\tilde{K})} \ ,&&0 \leq l \leq s \leq  p-1 , \ &&\forall \widehat{\boldsymbol{v}}  \in \widehat{\boldsymbol{V}}^4 \cap \mathcal{Y}_4(\p{H}^s(\Omega)).
\end{alignat*}
\end{lemma}

\begin{proof}
	The proof is completely analogous to the one of Lemma 5.1 in \cite{Buffa2011IsogeometricDD}.  We only show the assertion for the third inequality as an example since all the other inequalities follow with similar arguments. Let now $ \in \widehat{\boldsymbol{T}}  \in \widehat{\boldsymbol{V}}^3 \cap \mathcal{Y}_3(\p{H}^s(\Omega))$. Then, due to the smoothness of  $\p{F}$, we have $\widehat{\boldsymbol{T}}  \in \widehat{\boldsymbol{V}}^3 \cap \p{H}^s(\widehat{\Omega})$. Further let $0 \leq l \leq s \leq p-1$. In view of Lemma 3.1 in \cite{IGA3} we find  splines $\widehat{s}_1 \in S_{p-1,p-1,p-1}^{r-1,r-1,r-1}, \ \widehat{s}_2 \in S_{p-2,p,p-1}^{r-2,r,r-1}, \ \widehat{s}_3 \in S_{p-2,p-1,p}^{r-2,r-1,r}, \ \widehat{s}_4 \in S_{p,p-2,p-1}^{r,r-2,r-1}, \ \widehat{s}_5 \in S_{p-1,p-1,p-1}^{r-1,r-1,r-1}, \ \widehat{s}_6 \in S_{p-1,p-2,p}^{r-1,r-2,r},$ $ \ \widehat{s}_7 \in S_{p,p-1,p-2}^{r,r-1,r-2},  \ \widehat{s}_8 \in S_{p-1,p,p-2}^{r-1,r,r-2}, \ \widehat{s}_9 \in S_{p-1,p-1,p-1}^{r-1,r-1,r-1}$ with
	\begin{align}
	\label{eq:componentwise_estimate}
	|\widehat{s}_i - \widehat{T}(i)|_{\mathcal{H}^l(\tilde{K})} \leq C \ h^{s-l} |\widehat{T}(i)|_{\mathcal{H}^s(\tilde{K})}, \ \ \forall i, \ \ \ \widehat{\boldsymbol{T}} \eqqcolon \textup{MAT}(\widehat{T}(1),\dots ,\widehat{T}(9)).
	\end{align}
	Here we used the regular mesh assumption. This implies directly\\ $| \widehat{\boldsymbol{M}}  - \widehat{\boldsymbol{T}}|_{\boldsymbol{\mathcal{H}}^l(\tilde{K})} \leq C \ h^{s-l} |\widehat{\boldsymbol{T}}|_{\boldsymbol{\mathcal{H}}^s(\tilde{K})}, \ $ where $\widehat{\boldsymbol{M}} \coloneqq \textup{MAT}(\widehat{s}_1, \dots , \widehat{s}_9)$.
	Triangle inequality and the spline preserving property of the projections  lead to 
	\begin{align}
	\label{eq:Matrix_trinagle_est}
	|\widehat{\boldsymbol{T}}-\widehat{\Pi}^3_h\widehat{\boldsymbol{T}}|_{\p{H}^l(K)} \leq |\widehat{\boldsymbol{T}}-\widehat{\boldsymbol{M}}|_{\p{{H}}^l(K)} + | \widehat{\Pi}^3_h \big(\widehat{\boldsymbol{M}} - \widehat{\boldsymbol{T}} \big)|_{\p{{H}}^l(K)}.
	\end{align} The last term on the right-hand side can be estimated by means of the stability result for the projections (see Lemma \ref{lemma:stability_projections_param}) and a standard inverse estimate for polynomials. More precisely,
	\begin{align*}
	 | \widehat{\Pi}^3_h \big(\widehat{\boldsymbol{M}} - \widehat{\boldsymbol{T}} \big)|_{\p{H}^l(K)}  \leq C \, h^{-l}  \norm{ \widehat{\Pi}^3_h \big(\widehat{\boldsymbol{M}} - \widehat{\boldsymbol{T}} \big)}_{\p{L}^2({K})}  \leq C \, h^{-l}  \norm{ \widehat{\boldsymbol{M}} - \widehat{\boldsymbol{T}} }_{\p{L}^2(\tilde{K})}.
	\end{align*}
	The combination of the last estimate with the two lines \eqref{eq:componentwise_estimate} and \eqref{eq:Matrix_trinagle_est} leads to the third inequality in the assertion. All the other estimates are obtained similarly.
\end{proof}
Finally we can state the approximation properties of the projections $\Pi^i_h$ in the physical domain.

\begin{theorem}{(Approximation in the physical domain)}\\
	\label{theorem:approximation in the physical domain}
	There exist constants $C$ not depending on the mesh size, refinement respectively, such that
	\begin{alignat*}{3}
	& \norm{{\phi}-\Pi^1_h{\phi}}_{H^{l}(\Omega)} \leq C \ h^{s-l} \ \norm{{\phi}}_{{H}^s(\Omega)}, &&0 \leq l \leq s \leq  p+1 , \ &&\forall {\phi}  \in {V}^1 \cap H^s(\Omega),\\
	& \norm{{\p{S}}-\Pi^2_h{\p{S}}}_{\p{H}^{l}(\Omega)} \leq C \ h^{s-l} \ \norm{{\p{S}}}_{\boldsymbol{{H}}^s(\Omega)}, \hspace{0.5cm}&&0 \leq l \leq s \leq  p-1 , \ &&\forall {\p{S}}  \in \boldsymbol{V}^2 \cap \p{H}^s(\Omega),\\
	& \norm{{\p{T}}-\Pi^3_h{\p{T}}}_{\p{H}^{l}(\Omega)} \leq C \ h^{s-l} \ \norm{{\p{T}}}_{\boldsymbol{{H}}^s(\Omega)},&&0 \leq l \leq s \leq  p-1 , \ &&\forall {\p{T}}  \in \boldsymbol{V}^3 \cap \p{H}^s(\Omega), \\
	& \norm{{\boldsymbol{v}}-\Pi^4_h{\boldsymbol{v}}}_{\p{H}^{l}(\Omega)} \leq C \ h^{s-l} \ \norm{{\boldsymbol{v}}}_{\boldsymbol{{H}}^s(\Omega)},&&0 \leq l \leq s \leq  p-1 , \ &&\forall {\boldsymbol{v}}  \in \boldsymbol{V}^4 \cap \p{H}^s(\Omega).
	\end{alignat*}
\end{theorem}
\begin{proof}
	The proof is nearly the same as the one for Theorem 5.3. in \cite{Buffa2011IsogeometricDD} . For reasons of completeness we show the assertion for the third estimate. Let $K \in \widehat{{M}}$ be a mesh element and $\p{T} \in \boldsymbol{V}^3 \cap \p{H}^s(\Omega)$. Set $\widehat{\boldsymbol{T}} \coloneqq \mathcal{Y}_3(\p{T})$. Lemma \ref{lemma:regularity_transformation} yields for some constant $C$ the inequality
	\begin{align}
	\label{eq:approx_prrof_physical_1}
	\norm{\p{T}-\Pi^3_h\p{T}}_{\p{H}^l(\mathcal{K})} \leq C \, \norm{\widehat{\boldsymbol{T}}-\widehat{\Pi}^3_h\widehat{\boldsymbol{T}}}_{\p{H}^l(K)}.
	\end{align}
	Above we wrote $\mathcal{K} = \p{F}(K)$ and note that for $\p{T} \in \p{H}^s(\Omega)$ it is $\widehat{\boldsymbol{T}} \in \p{H}^s(\widehat{\Omega})$. Thus Lemma \ref{lemma:approximation_parametric_domain} gives $$\norm{\widehat{\boldsymbol{T}}-\widehat{\Pi}^3_h\widehat{\boldsymbol{T}}}_{\p{H}^l(K)} \leq C \, \sum_{i=0}^{l} |\widehat{\boldsymbol{T}}-\widehat{\Pi}^3_h\widehat{\boldsymbol{T}}|_{\p{H}^i(K)}  \leq   C \, h^{s-l} |\widehat{\boldsymbol{T}}|_{\boldsymbol{\mathcal{H}}^l(\tilde{K})}.$$
	The repeated application of Lemma \ref{lemma:regularity_transformation} , together with \eqref{eq:approx_prrof_physical_1}, leads to the third inequality of the assertion. All the other approximation estimates can be proven in a similar fashion.
\end{proof}

Exploiting the commutativity rules for the differential operators and projections we obtain a modified approximation estimate  which is useful for example in the subsequent section.

\begin{corollary}
	\label{corollary}
Let $1 \leq s \leq p-1$. Then, assuming sufficient regularity, it is
	\begin{alignat*}{3}
& \norm{{\phi}-\Pi^1_h{\phi}}_{H^{2}(\Omega)} \leq C \ h^{s} \ \norm{{\phi}}_{H^{s+2}(\Omega)}, \\
& \norm{{\p{S}}-\Pi^2_h{\p{S}}}_{\p{H}(\Omega,\textup{curl})} \leq C \ h^{s} \ \norm{{\p{S}}}_{\p{H}^s(\Omega,\textup{curl})}, \hspace{0.5cm}\\
& \norm{{\p{T}}-\Pi^3_h{\p{T}}}_{\p{H}(\Omega,\textup{div})} \leq C \ h^{s} \ \norm{{\p{T}}}_{\p{H}^s(\Omega,\textup{div})},\\
& \norm{{\boldsymbol{v}}-\Pi^4_h{\boldsymbol{v}}}_{\p{L}^{2}(\Omega)} \leq C \ h^{s} \ \norm{{\boldsymbol{v}}}_{\boldsymbol{H}^s(\Omega)}.
\end{alignat*}
We use the notation  $\norm{{\p{S}}}_{\p{H}^s(\Omega,\textup{curl})}^2 \coloneqq \norm{ {\p{S}}}_{\p{H}^s(\Omega)}^2 + \norm{ \nabla \times {\p{S}}}_{\p{H}^s(\Omega)}^2$ and $\norm{{\p{T}}}_{\p{H}^s(\Omega,\textup{div})}^2 \coloneqq \norm{ {\p{T}}}_{\p{H}^s(\Omega)}^2 + \norm{ \nabla \cdot {\p{T}}}_{\p{H}^s(\Omega)}^2$.	
\end{corollary}
\begin{proof}
	The estimates are a consequence of Theorem \ref{theorem:approximation in the physical domain} . We show the assertion only for the second line. The rest can be derived applying an analogous procedure.
	\begin{align*}
	\norm{{\p{S}}-\Pi^2_h{\p{S}}}_{\p{H}(\Omega,\textup{curl})}^2 &= \norm{{\p{S}}-\Pi^2_h{\p{S}}}_{\p{L}^2(\Omega)}^2 + \norm{\nabla \times {\p{S}}-\nabla \times \Pi^2_h{\p{S}}}_{\p{L}^2(\Omega)}^2 \\& =\norm{{\p{S}}-\Pi^2_h{\p{S}}}_{\p{L}^2(\Omega)}^2 + \norm{\nabla \times {\p{S}}- \Pi^3_h(\nabla \times {\p{S}})}_{\p{L}^2(\Omega)}^2 \\ 
	& \leq C \ \big( h^{2s} \, \norm{\p{S}}_{\p{{H}}^s(\Omega)}^2 + h^{2s} \norm{\nabla \times \p{S}}_{\p{{H}}^s(\Omega)}^2 \big) \leq C \ h^{2s} \ \norm{\p{S}}_{\p{H}^s(\Omega, \textup{curl})}^2.
	\end{align*}
 \end{proof}
After we focused on standard approximation inequalities for the introduced discrete spaces, we move on with examples for which such a FEEC based discretization is helpful.

\section{Applications}
In this part we show how one can apply the structure-preserving discretization to set up numerical schemes, following the results of FEEC (see Chapter 5 in \cite{ArnoldBook}),  for  different differential equations. Exploiting FEEC we will obtain  relatively easy stability and convergence statements. \\
We look at the case of the Hodge-Laplacians corresponding to the above studied Hessian complex. Thus let us first shortly define the class of the Hodge-Laplace equation. For more information on that we refer to \cite{ArnoldBook} and \cite{Arnold2010FiniteEE}.¸
\subsection{Hodge-Laplace equation}
For a given closed Hilbert-complex  $(W^k,d^k)$ with a domain complex $(V^k,d^k)$ we call the problem 
\begin{align}
\label{eq:Hodge_Lplaciab_level_k}
 \big( d_{k+1}^* d^k + d^{k-1} d_k^*)u = f
\end{align} (abstract) Hodge-Laplacian of level $k$, where $f \in W^k$ is some given mapping. Clearly, the operator $L^k \coloneqq d_{k+1}^* d^k + d^{k-1} d_k^*$ on the left-hand side has the domain $$D(L^k) = \{v \in V^k \cap {V_k^*} \ | \ d^kv \in{V_{k+1}^*}, \ d_k^*v \in V^{k-1} \}$$ and one requires $u \in D(L^k)$. 
Furthermore we assume from now on, in view of the Hessian complex from above, that the domain complex is exact.
Then, Theorem 4.8 in \cite{ArnoldBook} gives us the well-posedness of the problem \eqref{eq:Hodge_Lplaciab_level_k}.
For the computation of approximate solutions for the Hodge-Laplacian of level $k$ it is useful to introduce an equivalent (see Theorem 4.7 in \cite{ArnoldBook}) mixed weak formulation. It reads

\begin{alignat}{3}
\label{eq:mixed_weak_form}
\textup{Find \ }\ \sigma \in V^{k-1}, \ u \in V^k \ \textup{with} \nonumber\\
\langle \sigma, \tau \rangle - \langle u, d^{k-1} \tau \rangle  &= 0, \ &&\forall \tau  \in V^{k-1}, \\
\langle d^{k-1} \sigma,v \rangle + \langle d^ku,  d^kv \rangle  &= \langle f , v \rangle, \hspace{0.8cm} &&\forall v \in V^{k} \ . \nonumber
\end{alignat} 
The connection between the original formulation and the mixed weak form is given by $d_k^*u = \sigma$. Due to readability we neglected the indices $W^k$ for inner product brackets.\\   
The basic idea for a numerical method of the latter problem is to use test functions from  finite-dimensional subspaces $V_h^k \subset V^k$ instead of taking  the whole spaces $V^k$ into account. As in the case of spline spaces the index $h$ indicates refinement, dimension increase of the subspaces respectively, meaning for $h \rightarrow 0$ the dimension of the subspaces grow to infinity.
The work from  Arnold guides us to stable and convergent numerical methods, by requiring three key properties for the subspaces. In view of Theorems 4.8, 4.9 and 5.5 as well as  chapter 5 of \cite{ArnoldBook} we summarize the mentioned properties and results in the next Theorem.
\begin{theorem}{(Structure-preserving discretization of the Hodge-Laplacian)}\\
	\label{theorem:FEEC}
	Let $(W^k,d^k)$ be a closed Hilbert complex with exact domain complex $(V^k,d^k)$. Further let the next three properties be fulfilled:
	\begin{itemize}
		\item "Approximation property:" For $h \rightarrow 0$ we have $\underset{v \in V^i_h}{\inf}\norm{w-v}_{V^i} \rightarrow 0$, \ for all $w \in V^i$.
		\item "Subcomplex property": The subspaces form a subcomplex in the sense $d^{k-1}V_h^{k-1} \subset V_h^k, \ \  d^kV_h^k \subset V_h^{k+1}$. 
		\item "Bounded cochain projections": There exist $W^i$-bounded projections $ \Pi_h^{i} \colon V^{i} \rightarrow V_h^{i}$, which commute with the operators $d^i$. In other words the next diagram commutes.  	 \begin{figure}[h!]
			\begin{center}
				\begin{tikzpicture}
				\node at (0,1.4) {$V^{k-1}$};
				\node at (0,-0.4) {${V}^{k-1}_h$};
				\node at (4,1.4) {${V}^k$};
				\node at (4,-0.4) {${{V}}_h^k$};
				
				\node at (8.5,1.4) {${V}^{k+1}$};
				\node at (8.5,-0.4) {${{V}}_h^{k+1}$};
				
				\node at (9.5,-0.6) {$.$};
				
				\node at (-0.5,0.5) {$\Pi^{k-1}_h$};	
				
				\node at (3.7,0.5) {$\Pi^k_h$};	
				\node at (8,0.5) {$\Pi^{k+1}_h$};

				\node at (1.9,1.7) {$d^{k-1}$};
				
				\node at (1.9, -0.1) {$d^{k-1}$};
				\node at (6.2,1.7) {$d^k $};
				\node at (6.2,-0.1) {$d^k $};

				\draw[->] (0,1) to (0,0);
				\draw[->] (4,1) to (4,0);
				\draw[->] (8.5,1) to (8.5,0);
				
				\draw[->] (1,1.4) to (2.5,1.4);
				\draw[->] (1,-0.4) to (2.5,-0.4);
				
				\draw[->] (5.5,-0.4) to (7,-0.4);
				
				\draw[->] (5.5,1.4) to (7,1.4);

				\end{tikzpicture}
			\end{center}
		\end{figure}
		
	\end{itemize}

Then the discrete formulation of \eqref{eq:mixed_weak_form}, i.e.
\begin{equation}
\begin{alignedat}{3}
\label{eq:discrete_mixed_weak}
\textup{Find \ }\ \sigma_h \in V_h^{k-1}, \ u_h \in V_h^k \ \textup{with}  \\ 
\langle \sigma_h, \tau_h \rangle - \langle u_h, d^{k-1} \tau_h \rangle  &= 0, \ &&\forall \tau_h  \in V^{k-1}_h, \\ 
\langle d^{k-1} \sigma_h,v_h \rangle + \langle d^ku_h,  d^kv_h \rangle  &= \langle f , v_h \rangle, \hspace{0.8cm} &&\forall v_h \in V_h^{k},
\end{alignedat} 
\end{equation}

has a unique solution and satisfies an inf-sup stability criterion, namely

$$ \underset{0  \neq (\sigma_h,u_h) \in X_h^k }{\inf} \ \ \  \underset{0  \neq (\tau_h,v_h) \in X_h^k}{\sup} \  \frac{B(\sigma_h,u_h;\tau_h,v_h)}{\norm{(\sigma_h,u_h)}_{X^k} \, \norm{(\tau_h,v_h)}_{X^k}} \geq \gamma > 0, $$
with $ B(\sigma_h,u_h;\tau_h,v_h) \coloneqq\langle \sigma_h, \tau_h \rangle - \langle u_h, d^{k-1} \tau_h \rangle + \langle d^{k-1} \sigma_h,v_h \rangle + \langle d^ku_h,  d^kv_h \rangle$,\\ $X_h^k = V^{k-1}_h \times V^k_h$ and $\norm{(\tau_h,v_h)}_{X^k} \coloneqq \norm{\tau_h}_{V^{k-1}} + \norm{v_h}_{V^{k}}$.\\
Further, the discrete solution ($\sigma_h,u_h$) converges to the exact  ($\sigma$, $u$) one and it holds
$$ \norm{\sigma-\sigma_h}_{V^{k-1}} + \norm{u-u_h}_{V^k} \leq C \, \big(  \underset{\tau \in V^{k-1}_h}{\inf}\norm{\tau-\sigma}_{V^{k-1}} + \underset{v \in V^k_h}{\inf}\norm{u-v}_{V^{k}} \big), $$
where $C$ is a constant independent of $h$.
\end{theorem}

Now we use Theorem \ref{theorem:FEEC} for the Hodge-Laplacians with  underlying  Hessian complex and the above derived discretization of it. Namely, it is easy to see that the three important properties stated in latter theorem are satisfied by the discretization ansatz in Section 3. Therefore we apply again the notation of Sections 3 and 5, especially use \eqref{eq:notation_of_the spaces}. 
 Besides, below the mappings without index $h$ stand for the exact solutions of the respective mixed weak form; compare \eqref{eq:mixed_weak_form}.

\subsubsection*{Hodge-Laplacian: \ $k=1$}
We have the equation: 
\begin{align*}
\big( \iota_{P_1} \pi_{P_1}  + \nabla \cdot \nabla \cdot \nabla^2    \big) \phi = f,  \ \ \textup{with} \ \ \phi \in H^2({\Omega}),  \nabla^2 \phi \in {\overset{\circ}{\p{H}}(\Omega,\textup{divdiv}, \mathbb{S})}, \ {f} \in \ {L}^2(\Omega),
\end{align*}
and a corresponding discrete mixed weak formulation 
\begin{alignat*}{3}
\textup{Find \ }\ p_h \in P_1(\Omega), \ \phi_h \in V_h^1 \subset H^2({\Omega}) \ \textup{with}\\
\langle p_h, q_h \rangle - \langle \phi_h,  q_h \rangle  &= 0, \ &&\forall q_h  \in P_1(\Omega), \\
\langle  p_h, \tau_h \rangle + \langle \nabla^2 \phi_h,  \nabla^2 \tau_h \rangle  &= \langle f , \tau_h \rangle, \hspace{0.8cm} &&\forall \tau_h \in V_h^{1}.
\end{alignat*} 
Theorem \ref{theorem:FEEC} gives the stability and convergence of the method and an error estimate.
If we assume $1 \leq s \leq p-1$ and $\phi \in H^{s+2}(\Omega)$, then Corollary \ref{corollary} yields
\begin{align*}
\norm{p-p_h}_{L^2(\Omega)} + \norm{\phi-\phi_h}_{V^1} \leq C \, \big(  \underset{q_h \in V^{0}_h}{\inf}\norm{p-q_h}_{L^2(\Omega)} + \underset{\tau_h \in {V}^1_h}{\inf}\norm{\phi-\tau_h}_{H^2(\Omega)} \big) \\
\leq C \ h^{s} \  \norm{\phi}_{H^{s+2}(\Omega)} .
\end{align*}
For the last inequality we used the fact that $ P_1(\Omega) \subset V_h^0 $ and note $\norm{\phi}_{V^1}^2 \coloneqq \norm{\phi}^2_{L^2(\Omega)}+\norm{\nabla^2\phi}^2_{\p{L}^2(\Omega)} $.

\subsubsection*{Hodge-Laplacian: \ $k=2$}
The level $k=2$ Hodge-Laplacian can be written as
\begin{align*}
\big( \textup{sym}\nabla \times \nabla \times   + \nabla^2 \ \nabla \cdot \nabla \cdot  \, \big) \p{S} = \p{f}, \ \ \p{S} \in \p{H}({\Omega},\textup{curl}, \mathbb{S}), \ \ \p{f} \in \p{L}^2(\Omega,\mathbb{S}),
\end{align*}
where one searches for $\p{S} \in \p{H}({\Omega},\textup{curl}, \mathbb{S}) \cap{\overset{\circ}{\p{H}}(\Omega,\textup{divdiv}, \mathbb{S})} $ with $\nabla \times \p{S} \in {\overset{\circ}{\p{H}}(\Omega,\textup{symcurl}, \mathbb{T})}, \ $ $ \nabla \cdot\nabla \cdot \p{S} \in H^2(\Omega)$.
Thus the discrete mixed weak formulation in view  of \eqref{eq:discrete_mixed_weak}   is:
\begin{alignat*}{3}
\textup{Find \ }\ \phi_h \in V_h^{1} \subset H^2(\Omega), \ \p{S}_h \in \boldsymbol{V}_h^2 \subset \p{H}({\Omega},\textup{curl}, \mathbb{S}) \ \textup{with}\\
\langle \phi_h, \tau_h \rangle - \langle \p{S}_h, \nabla^2 \tau_h \rangle  &= 0, \ &&\forall \tau_h  \in V_h^1, \\
\langle \nabla^2 \phi_h, \p{M}_h \rangle + \langle \nabla \times \p{S}_h,  \nabla \times \p{M}_h \rangle  &= \langle \p{f} , \p{M}_h \rangle, \hspace{0.8cm} &&\forall \p{M}_h \in \boldsymbol{V}_h^{2}.
\end{alignat*} 
 Using again Theorem   \ref{theorem:FEEC} and  Corollary \ref{corollary}
  we obtain  stability and convergence with estimate
\begin{align*}
\norm{\phi-\phi_h}_{V^1} + \norm{\p{S}-\p{S}_h}_{\p{H}(\Omega,\textup{curl})} \leq C \, \big(  \underset{\tau_h \in V^{1}_h}{\inf}\norm{\tau_h-\phi}_{H^2(\Omega)} + \underset{\p{M}_h \in \boldsymbol{V}^2_h}{\inf}\norm{\p{S}-\p{M}_h}_{\p{H}(\Omega,\textup{curl})} \big) \\
\leq C \ h^{s} \  \big(\norm{\phi}_{H^{s+2}(\Omega)} +\norm{\p{S}}_{\p{H}^s(\Omega,\textup{curl})} \big).
\end{align*}
Above we assumed $1 \leq s \leq p-1$ and $\phi \in H^{s+2}(\Omega) \ \textup{as well as}  \  \p{S}, \nabla \times \p{S} \in \p{H}^s(\Omega)$.

\subsubsection*{Hodge-Laplacian: \ $k=3$}
\label{sec:k3}
In this case  equation \eqref{eq:Hodge_Lplaciab_level_k} has the  form
\begin{align*}
\big( -\textup{dev}\nabla \nabla \cdot   + \nabla \times ( \textup{sym}\nabla \times  )  \, \big) \p{T} = \p{f} , \ \ \p{f} \in \p{L}^2(\Omega,\mathbb{T}).
\end{align*}
And the strong form seeks $\p{T} \in \p{H}({\Omega},\textup{div}, \mathbb{T}) \cap {\overset{\circ}{\p{H}}(\Omega,\textup{symcurl}, \mathbb{T})}$ such that $\textup{sym}\nabla \times \p{T} \in {\p{H}(\Omega,\textup{curl}, \mathbb{S})}, \ \ \nabla \cdot \p{T} \in \p{H}_0^1(\Omega)$. The finite-dimensional problem based on the mixed weak form reads:
\begin{alignat*}{3}
\textup{Find \ }\ \p{S}_h \in \boldsymbol{V}_h^{2} \subset \p{H}({\Omega},\textup{curl}, \mathbb{S}), \ \p{T}_h \in \boldsymbol{V}_h^3 \subset \p{H}({\Omega},\textup{div}, \mathbb{T}) \ \textup{with}\\
\langle \p{S}_h, \p{M}_h \rangle - \langle \p{T}_h, \nabla \times \p{M}_h \rangle  &= 0, \ &&\forall \p{M}_h  \in \boldsymbol{V}_h^2, \\
\langle  \nabla \times \p{S}_h , \p{N}_h \rangle + \langle \nabla \cdot \p{T}_h,  \nabla \cdot \p{N}_h \rangle  &= \langle \p{f} , \p{N}_h \rangle, \hspace{0.8cm} &&\forall \p{N}_h \in \boldsymbol{V}_h^{3}.
\end{alignat*} 
Again if we require $1 \leq s \leq p-1$ and $\p{S}, \nabla \times \p{S} \in \p{H}^s(\Omega)$ and $ \ \p{T}, \nabla \cdot \p{T} \in  \p{H}^s(\Omega)$, then we get with Corollary \ref{corollary}:
\begin{align*}
\norm{\p{S}-\p{S}_h}_{\p{H}(\Omega,\textup{curl})} &+ \norm{\p{T}-\p{T}_h}_{\p{H}(\Omega,\textup{div})}\\ &\leq C \, \big(  \underset{\p{M}_h \in \boldsymbol{V}^{2}_h}{\inf}\norm{\p{M}_h-\p{S}}_{\p{H}(\Omega,\textup{curl})} + \underset{\p{N}_h \in \boldsymbol{V}^3_h}{\inf}\norm{\p{T}-\p{N}_h}_{\p{H}(\Omega,\textup{div})} \big) \\
&\leq C \ h^{s} \  \big(\norm{\p{S}}_{\p{H}^s(\Omega,\textup{curl})} +\norm{\p{T}}_{\p{H}^s(\Omega,\textup{div})} \big),
\end{align*}
where $\p{T}$ and $\p{S}$ are the exact solution of the mixed weak form of the  Hodge-Laplacian.

\subsubsection*{Hodge-Laplacian: \ $k=4$}
In view of $d^4 = 0$ we can write down the Hodge-Laplacian of level $4$ as
\begin{align*}
-\big( \nabla \cdot \textup{dev}\nabla      \, \big) \p{v} = \p{f}, \ \ \p{v} \in \p{L}^2({\Omega}), \ \ \p{f} \in \p{L}^2(\Omega),
\end{align*}
and we look for $\p{v} \in \p{H}_0^1(\Omega)$ with $ \textup{dev}\nabla \p{v} \in \p{H}(\Omega,\textup{div},\mathbb{T})$.
Now using the template \eqref{eq:discrete_mixed_weak} we obtain a finite-dimensional weak form:
\begin{alignat*}{3}
\textup{Find \ }\  \ \p{T}_h \in \boldsymbol{V}_h^3 \in \p{H}({\Omega},\textup{div}, \mathbb{T}), \  \p{v}_h \in \boldsymbol{V}_h^4 \subset \p{L}^2({\Omega}) \ \textup{with}\\
\langle \p{T}_h, \p{N}_h \rangle - \langle \p{v}_h, \nabla \cdot \p{N}_h \rangle  &= 0, \ &&\forall \p{N}_h  \in \boldsymbol{V}_h^3, \\
\langle  \nabla \cdot \p{T}_h , \p{w}_h \rangle  &= \langle \p{f} , \p{w}_h \rangle, \hspace{0.8cm} &&\forall \p{w}_h \in \boldsymbol{V}_h^{4}.
\end{alignat*} 
Applying Theorems \ref{theorem:FEEC} and Corollary \ref{corollary} yields, with
$1 \leq s \leq p-1$ and enough regularity, the inequality
\begin{align*}
\norm{\p{T}-\p{T}_h}_{\p{H}(\Omega,\textup{div})} + \norm{\p{v}-\p{v}_h}_{\p{L}^2(\Omega)} \leq C \, \big(  \underset{\p{N}_h \in \boldsymbol{V}^{3}_h}{\inf}\norm{\p{N}_h-\p{T}}_{\p{H}(\Omega,\textup{div})} + \underset{\p{w}_h \in \boldsymbol{V}^4_h}{\inf}\norm{\p{v}-\p{w}_h}_{\p{L}^2(\Omega)} \big) \\
\leq C \ h^{s} \  \big(\norm{\p{T}}_{\p{H}^s(\Omega,\textup{div})} +\norm{\p{v}}_{\p{H}^s(\Omega)} \big),
\end{align*}
where $\p{T}$ and $\p{v}$ are the  solution of the corresponding variational form \eqref{eq:mixed_weak_form}.\\

Another application for our structure-preserving discretization would be the Linearized Einstein-Bianchi (LEBS) system that is utilized to compute solutions to the linearized Einstein Field Equations in numerical relativity. For the next explanations we follow \cite{QuennevilleBlair2015ANA}.
\subsection{Linear Einstein-Bianchi system}
Using a suitable auxiliary variable and requiring consistent initial data the LEBS can be written as
\begin{align*}
\partial_t \sigma &= \nabla \cdot \nabla \cdot \p{E}, \\
\partial_t \p{E} &= -\nabla^2 \sigma - \textup{sym} \nabla \times \p{B}, \\
\partial_t \p{B} & = \nabla \times \p{E} , 
\end{align*}
with \begin{align*}
\sigma \in C^0([0,T];H^2(\Omega)) \cap C^1([0,T];L^2(\Omega)), \\ \p{E}\in C^0([0,T];\p{H}({\Omega},\textup{curl}, \mathbb{S}) \cap \overset{\circ}{\p{H}}(\Omega,\textup{divdiv}, \mathbb{S})) \cap C^1([0,T];\p{L}^2(\Omega)), \\ 
\p{B} \in C^0([0,T];\overset{\circ}{\p{H}}({\Omega},\textup{symcurl}, \mathbb{T}))  \cap C^1([0,T];\p{L}^2(\Omega)).
\end{align*}
 Above $0<T< \infty $ stands for some final time.

The spatial part comprises a Hodge-Laplacian and actually the upper equation belongs to the class of Hodge wave equations, the hyperbolic time-dependent version of  the Hodge-Laplacian; we refer to Chapter 8 in \cite{ArnoldBook}. In his thesis Quenneville-B\'elair \cite{QuennevilleBlair2015ANA} proofs an approximation estimate for a weak discrete formulation. More precisely, due to  Theorem 4.14 in \cite{QuennevilleBlair2015ANA}, if one has a structure-preserving spatial discretization in the sense of Theorem \ref{theorem:FEEC} of the Hessian complex (like the one we introduced in this article), then the solution of the problem 
\begin{align*}
\textup{Find \ }\ \sigma_h \in C^1([0,T]; V_h^{1}) \subset C^1([0,T];H^2(\Omega)), \hspace{5cm}\\ \ \p{E}_h \in C^1([0,T];\boldsymbol{V}_h^2) \subset C^1([0,T];\p{H}({\Omega},\textup{curl}, \mathbb{S})), \hspace{4cm} \\  \p{B}_h \in C^1([0,T];\boldsymbol{V}_h^3) \subset C^1([0,T];\p{H}({\Omega},\textup{div}, \mathbb{T})), \ \hspace{2cm} \textup{where} \hspace{1cm}
\end{align*}
\begin{alignat*}{3}
\langle \partial_t \sigma_h, \tau_h \rangle - \langle \p{E}_h, \nabla^2 \tau_h \rangle  &= 0, \ &&\forall \tau_h  \in V_h^1, \\
\langle \partial_t \p{E}_h, \p{M}_h \rangle +\langle \nabla^2 \sigma_h, \p{M}_h \rangle + \langle  \p{B}_h,  \nabla \times \p{M}_h \rangle  &= 0, \hspace{0.8cm} &&\forall \p{M}_h \in \boldsymbol{V}_h^{2},\\
\langle \partial_t \p{B}_h, \p{N}_h \rangle  - \langle \nabla \times \p{E}_h,  \p{N}_h \rangle  &= 0,
 \hspace{0.8cm} &&\forall \p{N}_h \in \boldsymbol{V}_h^{3}.
\end{alignat*} 
converges to the exact solution of the LEBS. And since there exists  a unique exact solution, we have a well-defined problem; compare \cite[Theorem 4.6., Theorem 4.14.]{QuennevilleBlair2015ANA}. One can proof that with proper initial data $(\sigma^0,\p{E}^0,\p{B}^0) \in V^1 \times (\p{V}^2 \cap \p{V}_2^*) \times \p{V}_3^*$ the difference between the exact  and the approximate solution satisfies
\begin{align*}
 \underset{t \in [0,T]}{\sup} \big(\norm{\sigma-\sigma_h}_{L^2(\Omega)} &+ \norm{\p{E}-\p{E}_h}_{\p{L}^2(\Omega)} + \norm{\p{B}-\p{B}_h}_{\p{L}^2(\Omega)}\big) \\ &\leq C \, \bigg(  \underset{t \in [0,T]}{\sup} \big(\norm{\sigma-\Pi_h^1\sigma}_{L^2(\Omega)} + \norm{\p{E}-\Pi_h^2\p{E}}_{\p{L}^2(\Omega)} + \norm{\p{B}-\Pi_h^3\p{B}}_{\p{L}^2(\Omega)}\big) \  \bigg) \\
 & \ \ + C \,  \int_{0}^{T} \norm{\dot{\sigma}-\Pi_h^1\dot{\sigma}}_{V^1} + \norm{\dot{\p{E}}-\Pi_h^2\dot{\p{E}}}_{\boldsymbol{V}^2} + \norm{\dot{\p{B}}-\Pi_h^3\dot{\p{B}}}_{\boldsymbol{V}^3} \ dt\\& \in  \mathcal{O}(h^{s}) ,
\end{align*}
assuming $1 \leq s \leq p-1$ and enough regularity of the exact solution \footnote{We have to assume that the exact solution variables are  elements of $C^1([0,T];V^1), \ C^1([0,T];\p{V}^i)$ respectively. }. 
In the second last line the dots denote the time derivative.

 We want to remark that in \cite{QuennevilleBlair2015ANA}  also other formulations and methods  are established and considered. Main reason for this is to simplify the discretization procedure, since the $H^2$-regularity is quite hard to achieve in the context of classical FEM.\\

We proceed now with the consideration of some numerical test problems to check the theoretical findings from above.

\section{Numerical examples}
In this part we want to display some numerical examples which consist of the computation of approximate solutions to the Hodge-Laplacians to verify the convergence statements.
Basis of the calculations is the GeoPDEs package (\cite{geopdes,geopdes3.0}) together with the MATLAB software; see \cite{MATLAB:2019}. The appearing linear systems are mainly solved by means of the mldivide function of MATLAB. Nevertheless for some very large systems we utilized the \emph{minres}-solver, where a  large number of iterations ($40000$) and small residuals, compared to the computed actual errors, ensure meaningful results.\\
As computational domain we consider a deformed cube obtained by the parametrization map $$\p{F} \colon (0,1)^3 \rightarrow \mathbb{R}^3 \ , \  \boldsymbol{\zeta} \mapsto A \cdot \boldsymbol{\zeta}, \  \ \ A \coloneqq \begin{pmatrix}
1 & 0.5 & 0.5 \\ 0 & 1 & 0.5 \\ 0.5 & 0 & 1
\end{pmatrix}.$$
For our experiments we divide the parametric reference cube into equal smaller cubes with edge lengths $h$ to generate regular meshes in the physical domain. In Fig.  \ref{fig:mesh} the mesh  with  size $h = 1/4$ is displayed.

Now we go through the cases of $k = 1,  \dots, 4$ and compute numerical solutions for the Hodge-Laplacian by means of the discrete mixed weak formulation \eqref{eq:discrete_mixed_weak}. In all test examples we use right-hand sides such that the exact solutions are smooth. Hence looking at Theorem \ref{theorem:FEEC} and Corollary \ref{corollary}, we expect for the different graph norms $\norm{\cdot}_{V^k}$,  $\norm{\cdot}_{\boldsymbol{V}^k}$ a convergence order one lower than the chosen degree $p$ for the space $V_h^1$. The polynomial degree $p$ in $V_h^1$ is chosen to be equal w.r.t. each parametric coordinate, i.e. $p_i=p$, and the splines have the regularities $r_i = r \coloneqq  p-1, \ \forall i$. Latter is done in order to save computational costs.
\subsubsection*{Hodge-Laplacian: \ $k=1$}
The right-hand side $f$ is adapted in such a way  that $\phi \circ \p{F} = \sin(\pi \ \zeta_1)^4 \ \sin(\pi \ \zeta_2)^4 \ \sin(\pi \ \zeta_3)^4$ is the exact solution. Obviously, we get $\phi \in D(L^1)$. 
If we compute the errors between numerical and exact solution for polynomial degrees $p = 2,3,4$ we obtain the results displayed in Fig. \ref{fig:conv_h2}.  The convergence behavior in the mentioned figure Fig. \ref{fig:conv_h2} confirms the theoretical assertion.

\begin{figure}[h!]
	\begin{minipage}{7cm}
		\includegraphics[width=7cm,height=5cm]{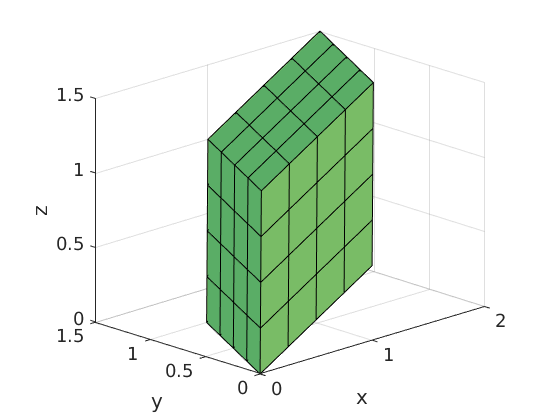}
		\caption{\small Physical domain we use for all test examples.}
			\label{fig:mesh}
	\end{minipage}
	\hspace{0.8cm}
	\begin{minipage}{7cm}
		\includegraphics[width=7cm,height=5cm]{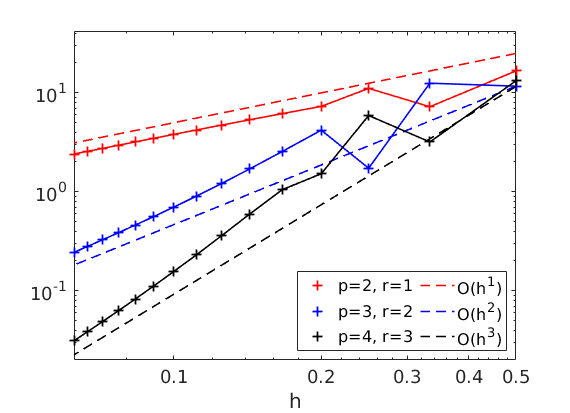}
		\caption{The errors $\norm{\phi-\phi_h}_{H^2}$ for the level $k=1$ Hodge-Laplacian.}
		\label{fig:conv_h2}
	\end{minipage}
\end{figure}

\subsubsection*{Hodge-Laplacian: \ $k=2$}
Here we set the source term $\p{f}$ s.t. the  exact solution has the form $$\p{S} = \textup{MAT}(v,v,\dots,v), \ \ v \coloneqq (\sin(\pi \ \zeta_1)^2 \ \sin(\pi \ \zeta_2)^2 \ \sin(\pi \ \zeta_3)^2) \circ \p{F}^{-1}(x,y,z).$$ One notes the  small deviation between  the predicted convergence order $1$ for $p=2$ and the plot in Fig. \ref{fig:conv_k1_hcurl}, where the errors in the norm $\norm{\cdot}_{\p{H}(\Omega,\textup{curl})}$ are plotted. Nevertheless,  there is no contradiction to Theorem \ref{theorem:FEEC}. And in Fig. \ref{fig:conv_k1_h2} we also see the errors $\norm{d_2^*\p{S}-\phi_h}_{H^2}$ which decay steadily in good accordance with the theory.

\begin{figure}[h!]
	\begin{minipage}{7cm}
		\includegraphics[width=7cm,height=5cm]{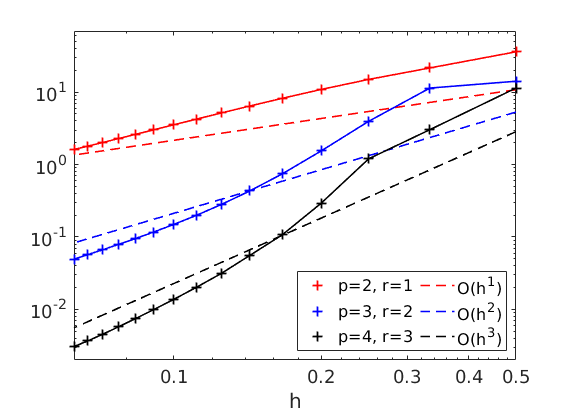}
		\caption{\small $H(\Omega,\textup{curl})$-error for the level $k=2$.}
		\label{fig:conv_k1_hcurl}
	\end{minipage}
	\hspace{0.8cm}
	\begin{minipage}{7cm}
		\includegraphics[width=7cm,height=5cm]{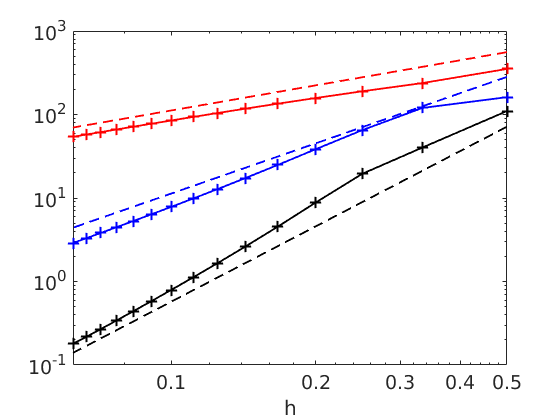}
		\caption{ \small The errors $\norm{d_2^*\p{S}-\phi_h}_{H^2}$ ($k=2$).}
		\label{fig:conv_k1_h2}
	\end{minipage}
\end{figure}

\subsubsection*{Hodge-Laplacian: \ $k=3$}
As above we use a manufactured exact solution to study the convergence behavior. \\ We have as exact solution $$\p{T}= \begin{pmatrix}
v & v & v \\ 0 & v & 0 \\ 0 & 0 & -2 \, v
\end{pmatrix}.  $$ In view of Theorem \ref{theorem:FEEC} and Section \ref{sec:k3}, we compute the errors $\norm{\p{T}-\p{T}_h}_{\p{H}(\Omega,\textup{div})}$ and \\ $\norm{\p{S}-\p{S}_h}_{\p{H}(\Omega,\textup{curl})}$, with $\p{S}= \textup{sym}\nabla \times \p{T}$. One can see the results in Fig. \ref{fig:conv_k2_hdiv} and Fig. \ref{fig:conv_k2_hcurl}.   
\begin{figure}[h!]
	\begin{minipage}{7cm}
		\includegraphics[width=7cm,height=5cm]{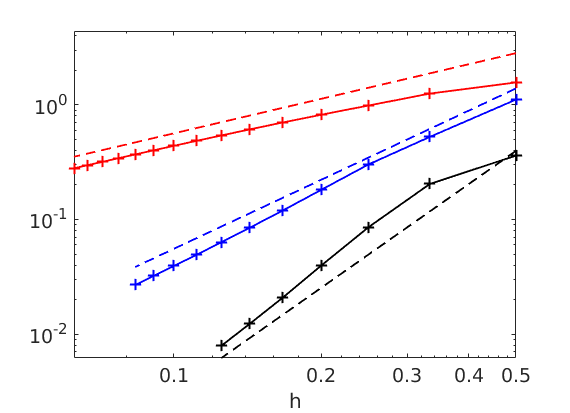}
		\caption{\small $H(\Omega,\textup{div})$-error for  level $k=3$.}
		\label{fig:conv_k2_hdiv}
	\end{minipage}
	\hspace{0.8cm}
	\begin{minipage}{7cm}
		\includegraphics[width=7cm,height=5cm]{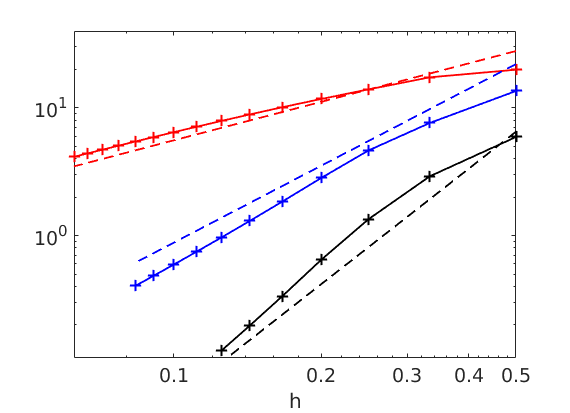}
		\caption{ \small The errors $\norm{d_3^*\p{T}-\p{S}_h}_{H(\text{curl})}$.}
		\label{fig:conv_k2_hcurl}
	\end{minipage}
\end{figure}

\subsubsection*{Hodge-Laplacian: \ $k=4$}
For our last level in the context of the Hodge-Laplacian, we use as test case the exact solution
$\p{v} = (v,v,v)$. The  errors $\norm{\p{v}-\p{v}_h}_{\p{L}^2(\Omega)}$ and $\norm{\p{T}-\p{T}_h}_{\p{H}(\Omega,\textup{div})}$ between exact and numerical solution are shown in the figures Fig. \ref{fig:conv_k3_L2} and Fig.\ref{fig:conv_k3_hdiv}. One notes the relation $\p{T} = -\textup{dev}\nabla\p{v}$.

\begin{figure}[h!]
	\begin{minipage}{7cm}
		\includegraphics[width=7cm,height=5cm]{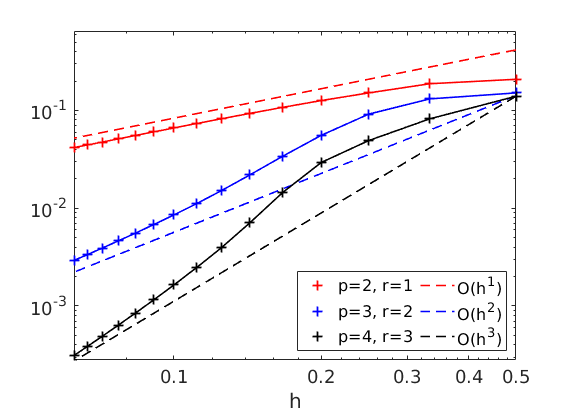}
		\caption{\small $L^2$-error for the level $k=4$.}
		\label{fig:conv_k3_L2}
	\end{minipage}
	\hspace{0.8cm}
	\begin{minipage}{7cm}
		\includegraphics[width=7cm,height=5cm]{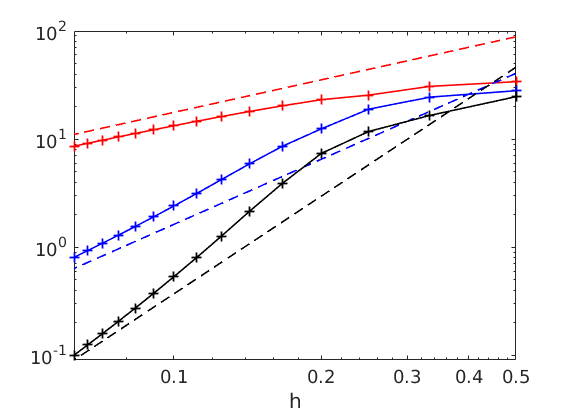}
		\caption{ \small The $H(\Omega,\textup{div})$-errors for   $k=4$.}
		\label{fig:conv_k3_hdiv}
	\end{minipage}
\end{figure}

In another small example we want to point out the superiority of the structure-preserving ansatz in contrast to straight-forward IGA discretizations. For this purpose we consider again the Hodge-Laplacian test cases for $k=2$ and $k=3$ from above. But now we use for each occurring component function  the standard IGA test spaces $V^1_h$ and ignore the other constructed finite-dimensional spaces.  The results for the errors and different mesh sizes are shown in the figures Fig. \ref{fig:conv_k1_nons_h2}-\ref{fig:conv_k2_nons_hdiv}. One sees the non-stable decay behavior for the errors and we interpret it as a clue for possible instability problems if one devotes oneself with  classical test function spaces.  In particular we can not guarantee that the error decays match with the theoretically predicted rates. For example in Fig. \ref{fig:conv_k2_nons_hdiv}, we observe partly an error increase despite mesh refinement.

\begin{figure}[h!]
	\begin{minipage}{7cm}
		\includegraphics[width=7cm,height=5cm]{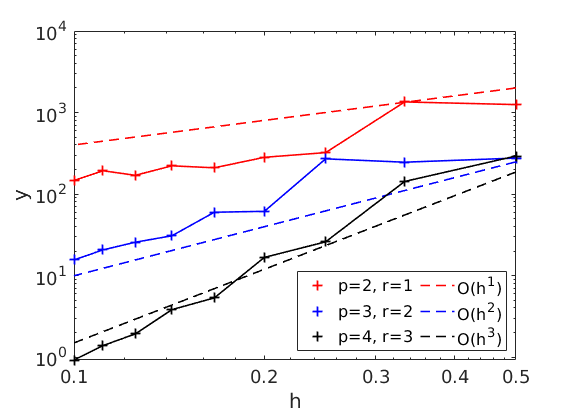}
		\caption{\small $H^2$-error for $k=2$ and a non-structure-preserving discretization.}
		\label{fig:conv_k1_nons_h2}
	\end{minipage}
	\hspace{0.8cm}
	\begin{minipage}{7cm}
		\includegraphics[width=7cm,height=5cm]{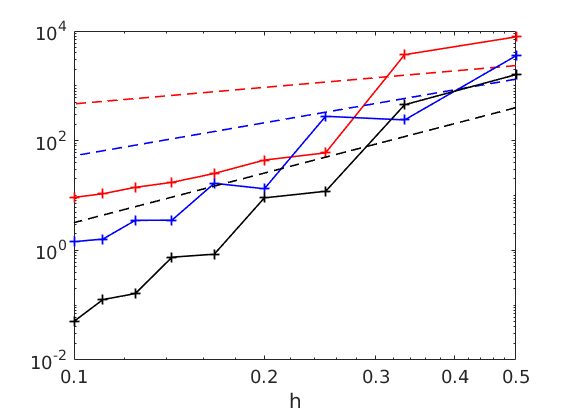}
		\caption{ \small $H(\Omega,\textup{curl})$-error for $k=2$ and a non-structure-preserving discretization.}
		\label{fig:conv_k1_nons_hcurl}
	\end{minipage}
\end{figure}

\begin{figure}[h!]
	\begin{minipage}{7cm}
		\includegraphics[width=7cm,height=5cm]{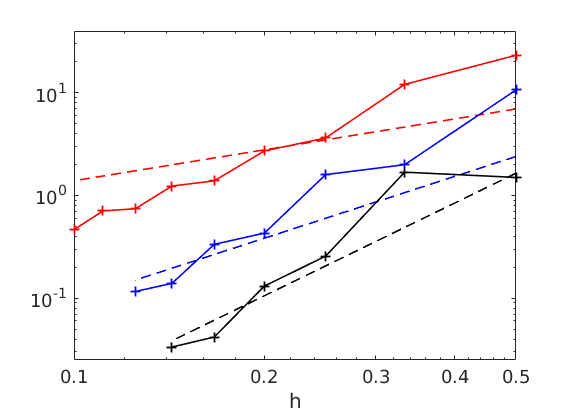}
		\caption{\small $H(\Omega,\textup{curl})$-error for $k=3$ and a non-structure-preserving discretization.}
		\label{fig:conv_k2_nons_hcurl}
	\end{minipage}
	\hspace{0.8cm}
	\begin{minipage}{7cm}
		\includegraphics[width=7cm,height=5cm]{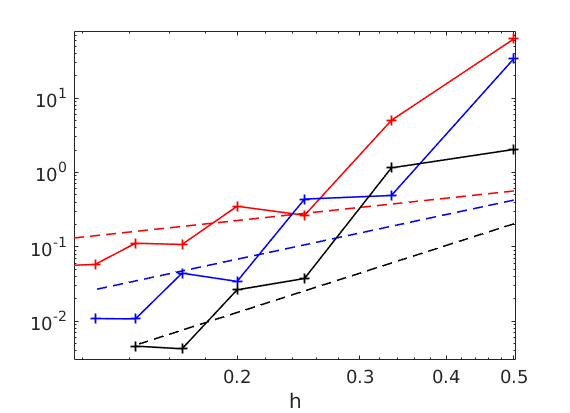}
		\caption{ \small $H(\Omega,\textup{div})$-error for $k=3$ and a non-structure-preserving discretization.}
		\label{fig:conv_k2_nons_hdiv}
	\end{minipage}
\end{figure}
So we can conclude for this section that the test examples confirm the statements in Theorem \ref{theorem:FEEC}, the structure-preserving property of our  discretization respectively.

\section{Discussion and further problems}
We introduced  spline spaces which are suitable for a discretization of the Hessian complex, since they mimic the complex structure. In other words, they can be used to built up a sub-complex. Then the theory of FEEC shows us how to write down mixed-weak formulations of the Hodge-Laplacian problem that guarantee stability and convergence. Further, the discretization can also be used in other fields like numerical relativity. And here we want to remark, that the symmetry and trace properties as parts of the original Hessian complex are preserved with our method exactly. But, there are also questions  and new problems arising. On the one hand, we could only show the  meaningfulness of the transformation mappings $\mathcal{Y}_i$ in case of affine linear parametrizations of the physical domain. Clearly, it is a reasonable thought to check for possible generalizations to exploit the benefit of IGA, namely the exact representation of curved boundary domains. The authors have already addressed the issue of curved boundary geometries and although one can show the existence of generalized structure-preserving transformations for non-trivial geometries we have to study the problem in more detail to obtain practicable outcomes. On the other hand, as a task for further studies it would be natural to consider other second order complexes in a similar fashion, e.g. the \textup{divdiv}-complex. Furthermore, there exists also a Hessian complex involving Dirichlet boundary conditions. Hence the construction of proper spline spaces satisfying the boundary conditions is also an interesting problem. 

\newpage

\nocite{*}
\bibliographystyle{siam}
\bibliography{Literatur}
\end{document}